\documentclass{article}
\usepackage{amsmath,amsthm}
\usepackage{amsfonts}
\newtheorem{proposition}{Proposition}

\newtheorem{definition}{Definition}
\newtheorem{theorem}{Theorem}
\newcommand{\cc}{\mathbb{C}}
\newcommand{\rr}{\mathbb{R}}
\newcommand{\hh}{\mathbb{H}}
\newcommand{\te}{{\mathbb{T}}}
\newcommand{\tes}{{{\mathbb{T}}^L}}
\newcommand{\tesR}{{{\mathbb{T}}^R}}
\newcommand{\oo}{\mathbb{O}}
\newcommand{\BB}{\mathbb{B}}
\newcommand{\ip}{{\,\rm i\,}}
\newcommand{\ipp}{{\,\rm J\,}}

\newcommand{\w}{{\rm w}}
\newcommand{\z}{\mathbb{Z}}

\newcommand{\nn}{\nonumber}
\newcommand{\ba}{\left[\begin{array}}
\newcommand{\ea}{\end{array}\right]}
\newcommand{\cdote}{\!\cdot\! }

\begin{document}

\pagestyle{myheadings} \markright{\hfill L. A. Wills-Toro: Classification of some graded division algebras I\ \ \ \ \ \ \ }

\begin{titlepage}
\title{{\bf Classification of some graded\\ not necessarily associative division algebras I}}

\author{Luis Alberto {\bf Wills-Toro}$^{1,2}$\\
~\\
1. School
 of Mathematics, Universidad Nacional de Colombia,\\ Medell\'{\i}n,
 Colombia, e--mail: lawillst@unal.edu.co\\
2. Dept. of Math. and Stat.,
 American University of
Sharjah,\\ P.O.Box 26666, Sharjah
 U.A.E.}
\date{}
\maketitle

\thispagestyle{empty}

\begin{abstract}
We study  not necessarily associative (NNA) division algebras over the reals. We classify in this paper series
those that admit a grading over a finite group
$G$, and have a basis $\{v_g|g\in G\}$ as a real vector space, and
the product of these basis elements respects the grading and
includes a scalar structure constant with values only in $\{1,-1\}$.
We classify here those graded by an abelian group
$G$ of order $|G|\leq 8$ with $G$ non--isomorphic to $\z/8\z$. We will find the complex,
 quaternion, and octonion algebras, but also a remarkable set of
novel non--associative division algebras.

~
{\bf AMS-MSC: 17A35, 17A01, 16S35, 17B30. }
\end{abstract}
\end{titlepage}

 \pagenumbering{arabic}
\section{Introduction}

We initiate the classification of some
not necessarily associative (NNA) division algebras over $\rr$.
We recover in the process the complex, quaternion, and
octonion algebras a set of novel
NNA division $\rr$--algebras. We use $\z_n:=\z/n\z$.

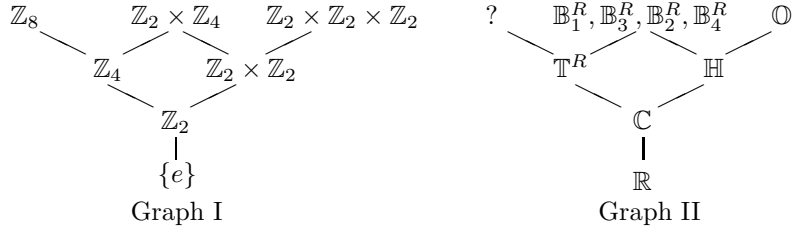
\begin{center}
\setlength{\unitlength}{1cm}
\begin{picture} (5.5,3)
\put(0.4,2.6){$\z_8$}
\put(2,2.6){$\z_2\times\z_4$}
\put(3.8,2.6){$\z_2\times\z_2\times\z_2$}
\put(0.8,2.5){\line(2,-1){.6}}
\put(2.5,2.5){\line(-2,-1){.6}}
\put(2.7,2.5){\line(2,-1){.6}}
\put(4.4,2.5){\line(-2,-1){.6}}
\put(1.5,1.9){$\z_4$}
\put(3,1.9){$\z_2\times\z_2$}
\put(1.7,1.8){\line(2,-1){.6}}
\put(3.4,1.8){\line(-2,-1){.6}}
\put(2.4,1.2){$\z_2$}
\put(2.6,0.8){\line(0,1){.3}}
\put(2.3,0.5){$\,\{e\}$}
\put(2,0){Graph I}
\end{picture}
\begin{picture} (6.5,3)
\put(1.1,2.6){$?$}
\put(2.0,2.6){$\BB_1^R,\BB_3^R,\BB_2^R,\BB_4^R$}
\put(4.9,2.6){$\oo$}
\put(1.4,2.5){\line(2,-1){.6}}
\put(3.1,2.5){\line(-2,-1){.6}}
\put(3.3,2.5){\line(2,-1){.6}}
\put(4.9,2.5){\line(-2,-1){.6}}
\put(2.0,1.9){$\tesR$}
\put(4,1.9){$\hh$}
\put(2.3,1.8){\line(2,-1){.6}}
\put(4,1.8){\line(-2,-1){.6}}
\put(3.0,1.2){$\,\cc$}
\put(3.2,0.8){\line(0,1){.3}}
\put(3.0,0.4){$\,\rr$}
\put(2.6,0){Graph II}
\end{picture}
\end{center}
Graph II presents some of the findings of this paper. With the exception of $\z_8$, the abelian groups in Graph I provide the grading for the NNA division algebras we consider here. It includes each group and their largest proper subgroups (lower included in the upper when connected by a bar). The Graph II presents the corresponding  isomorphism classes of finite dimensional NNA division $\rr$--algebras  and their maximal proper subalgebra relations (lower included in the upper when connected by a bar). Again, although multiple subalgebras might be included, it illustrates only their isomorphism classes. The typology of the classified finite dimensional  NNA division $\rr$--algebras in Graph II is constrained to those  algebras ${\cal A}$ that admit a grading over a finite group
$G$, and have a basis $\{v_g|g\in G\}$ as a real vector space, and
the product of these basis elements respects the grading and
includes a scalar structure constant $C_{\cal A}$ with values only in $\{1,-1\}$. Such $\rr$--algebras will be called twisted group algebras $({\cal A}; G,\rr,\{1,-1\},C_{\cal A})$. With the exception of $\BB_1^R$ and $\BB_3^R$, all the algebras in Graph II are $\rr$--algebra isomorphic to their own opposite algebras. $\BB_1^R$ is isomorphic to the opposite algebra of $\BB_3^R$. The NNA division algebras $({\cal A}; G,\rr,\{1,-1\},C_{\cal A})$ graded by nonabelian groups of order 8 are addressed in \cite{PapII}. Those graded by $\z_8$ (question mark in Graph II) are addressed in \cite{PapIII}.

\begin{definition} A unital not--necessarily--associative (NNA)  ring $R$
with multiplicative identity element $1$ is
called a {\bf NNA division ring} if for every $v\neq0,\
v\in R$ the
left--multiplication  map $x\mapsto v\cdot x$, and the
right--multiplication map $x\mapsto x\cdot v$  are bijections.
 If this NNA division ring with its summation and product is also an algebra over a field
$K$ we have a {\bf NNA
 division $K$--algebra}. The NNA structures are called {\bf nonassociative} with the understanding that they include both  associative and  not associative cases \cite{Schafer}.\end{definition}

\begin{definition} Let $A\subset K^*=K-\{0\}$ be a multiplicative subgroup
of $K^*$, with  $K$ a field, and let $G$ be a finite group of order $|G|$ with identity
element $e$. Given a function
$ C: G\times G\to A\subset K^*$
 we call $C$ a {\bf structure constant  of $G$ in $A$}.  $C$ is {\bf
unital} if additionally
$ C(e,g)=C(g,e)=1\ \forall g\in G$. We present frequently the structure constant as an array of
numbers in $A$, with matrix labels in $G$.
\end{definition}

\begin{definition} Let $C$ be a unital
structure constant of $G$ in $A$. We define a {\bf $C$--loop extension of $G$}, denoted $A\times_C G$,  as the set $A\times G$ endowed with
the not--necessarily--associative binary operation (product):
\begin{equation} (\alpha,g)\cdot (\beta,h)=(\alpha\beta\,
C(g,h),\,gh).\label{(2o)}\end{equation}\end{definition}

Every $C$--loop extension is a loop (quasi-group with identity
$(1,e)$). We will study here the case where $K=\rr$ and
$A=\{1,-1\}\cong \z/2\z\equiv \z_2$. In this case, the $C$--loop extension  $A\times_C G$  has $2|G|$ elements, and  it constitutes a
non--necessarily--associative extension of discrete finite
groups. The representation theory of $C$--loop extensions and
some not--necessarily--associative semisimple algebras leads to
the study and classification of a certain kind of
  not--necessarily--associative  division algebras, which will be the
scope of this  paper series.

We are interested on classifying
some twisted group $K$--algebras with finite grading group with particular properties.

\begin{definition} Let $G$ be a finite group with identity
$e$. A {\bf twisted group $K$--algebra ${\cal A}$} is a NNA $K$--algebra with unit $1\in {\cal A}$ that
as a $K$--vector space:
\begin{equation} {\cal A}=\bigoplus_{g\in G} W_g\ \
{\rm where}\ \ \dim_K W_g=1,\ {\rm and\ }\ W_g\cdot W_h\subset W_{gh},\nn
\end{equation}
and for every choice of base elements $v_g\in W_g$ for each $g\in G,\  g\neq e$, and $v_e=1$, there exists a  unital structure constant $C$ for $G$ in $K^*$, so that
$ v_g\cdot v_h=C(g,h)v_{gh}$, for all $g,h\in G$.
We present now a definition that emphasizes the existence of structure constants with remarkable characteristics:\\
Let $C$ be a unital structure constant of $G$ in $A\subset
K^*$. Let ${\cal A}=K_C\,G$ be a unital graded  NNA
algebra over a field $K$, which has a basis $\{v_g:
g\in G\}$ as a $K$--vector space, and  a ({\bf $C$-twisted}) product that
extends bi-linearly from
\begin{equation} v_a\cdot v_b=
C(a,b)\; v_{ab}, \ \ \forall a,b\in G.\label{(0)}\end{equation}
We
call such  an algebra  ${\cal A}=K_C\,G$ or $({\cal A}; G,K, A,C)$
a {\bf twisted group algebra over $K$}. Clearly, the
multiplicative neutral element of ${\cal A}$ is $v_e$.
$G$ is called the {\bf grading group of} $K_C\,G$.
In the case where $C$ is constant, that is $C(a,b)=1,\  \forall
a,b\in G$, the algebra $K_C\,G$ coincides with the group algebra
$KG$.\end{definition}

Our definition highlights the existence of a basis for which the structure
constant has particular properties: it takes values in $A\subset K^*$. In the case $K=\rr$ we observe that $\{1,-1\}$ is the only finite subgroup of  its multiplicative group $\rr^*=\rr-\{0\}$.
Concretely, we want to classify the twisted group algebras $({\cal
A}; G,K, A,C)=({\cal A}; G,\rr, \{1,-1\},C)$ which are
NNA  division $\rr$--algebras.
We have thus a finite grading group $G$ and a basis in which the
structure constant entries can only take values in
$\{1,-1\}$. This classification is a finite problem, since there
are only finitely many structure constant choices to make.

We remark also that although the twisted group algebra has a natural grading
structure, we will classify the algebras using  {\bf plain $\rr$--algebra isomorphisms} as far as we can.
There are  of course graded algebra isomorphisms which are more restrictive since they need to respect fixed grading assignments.

H. Hopf proved in \cite{Hopf} that every NNA
finite--dimensional division
algebra over $\rr$ has dimension a power of 2. We will
reproduce in an elementary fashion this result for the twisted
group algebras we are considering. M. Kervaire  in \cite{Kervaire} and
R. Bott and J. Milnor  in \cite{BottMilnor} independently proved that
every not--necessarily--associative finite--dimensional division algebra
over the reals
has dimensions 1, 2, 4, or 8. For some recent developments on  NNA
division algebras and twisted group algebras see  \cite{Pumpluen}, \cite{Darpo},
 \cite{PumplAst}, \cite{Darpo1}, \cite{AHK}, \cite{WC}, \cite{GMBandJMO},  \cite{Amitsur}, \cite{Rowen}, \cite{Albert}, \cite{Dickson}, \cite{Karrer},
\cite{EdLew}, \cite{EdLew1}, \cite{Passman}, \cite{MollNach},  and \cite{Bales}. We complement  with this paper series some explorations using some generalizations of the Cayley--Dickson doublings in \cite{WC}, \cite{PumplAst}, \cite{Pumpluen} in the case where the NNA division $\rr$--algebra is a twisted
group algebra  $({\cal A}; G,\rr, \{1,-1\},$ $C)$. Our explorations has  no constraint on the  type of generalized Cayley--Dickson doubling that might arise.
We will single out some NNA division $\rr$--algebras among the ones already found by other authors, and we study some of their astonishing properties. But we will also find some  novel NNA division $\rr$--algebras which are not among the families previously studied.

In this paper we  explore the zero divisors in twisted
group algebras  $({\cal A}; G,\rr$, $\{1,-1\}, C)$, and  define standard bases,  diagonal conjugations and involutions
for those algebras as
$\rr$--vector spaces. Then  we
proceed by the order of the grading group classifying the twisted group
algebras $({\cal A}; G,\rr, \{1,-1\},$ $C)$ which are NNA division $\rr$--algebras with
abelian grading group $G\not\simeq\z_8$.

\section{Zero divisors in twisted group algebras}
Let $x,y\in ({\cal A};G,\rr ,\{1,-1\},C)$ with basis $\{v_g:
g\in G\}$,
\begin{equation} x=\sum_{a\in H\subset G}x_a\,v_a,\ \ y=\sum_{b\in
H\subset G}y_b\,v_b,\label{SubAlgEls}\end{equation} where $H$ is a
subgroup of $G$. We call the real coefficients $x_a$ and $y_a$ the
{\bf components} of $x$ and $y$ respectively.
The product $x\cdot y$ can be rewritten as:
\begin{equation} x\cdot y =\sum_{c,a\in H}[(C(a,a^{-1}c)\,y_{a^{-1}c})\,x_a]\,v_c
=\sum_{c,b\in
H}[(x_{cb^{-1}}\,C(cb^{-1},b))\,y_b]\,v_c.\label{ProdXY}
\end{equation}
Hence, the set of elements in
${\cal A}$ of the form (\ref{SubAlgEls}) for $H$ subgroup of $G$
constitute automatically a subalgebra ${\cal B}$ which is $H$--graded.
Clearly,  a zero divisor in a sub\-algebra implies a zero divisor in
the whole algebra.
From  (\ref{ProdXY}), there are zero
divisors in ${\cal B}$ if we obtain non--trivial (component) solutions to the
systems
\begin{equation}\sum_{a\in H}(R_{{\cal B}}(y))_{c,a}\,x_a=0,\
\sum_{b\in
H}(L_{{\cal B}}(x))_{c,b}\,y_b=0, \ \forall c\in H\subset G\nn\end{equation}
 where the matrices $R_{{\cal B}}(y)$ and $L_{{\cal B}}(x)$ (whose subindex ${\cal B}$ will be  omitted when ${\cal B}={\cal A}$) are respectively associated to the product by $y$ from the left and the product by $x$ from the right in the subalgebra ${\cal B}$. In components:  $\forall a,b,c\in H$
\begin{equation}(R_{{\cal B}}(y))_{c,a}:=C(a,a^{-1}c)\,y_{a^{-1}c},\ \ {\rm and}\ \
(L_{{\cal B}}(x))_{c,b}:=x_{cb^{-1}}\,C(cb^{-1},b).\label{DetMRML}\end{equation}
There are thus zero divisors in ${\cal B}$ if either there is a
solution $y\neq 0$ to $\det(R_{{\cal B}}(y))=0$ or there is a solution $x\neq 0$ to $\det(L_{{\cal B}}(x))=0$.
The determinants of the matrices in (\ref{DetMRML})   are homogeneous polynomials of order $|H|$ in
the non--zero components involved, and they are of order $|H|$ in each of the nonzero
components.  We adopt $y_g=x_g=1$
for all $g\neq e,\ g\in H$, and then the determinants  of $R_{{\cal B}}(y)$ and $L_{{\cal B}}(x)$ in
(\ref{DetMRML}) become polynomials of order $|H|$ in a single variable
$y_e$ and $x_e$ respectively. The coefficients for
the monomials  $y_e^{|H|}$ and $x_e^{|H|}$ have
absolute value $|\prod_{a\in H} C(a,e)|=|\prod_{b\in H}
C(e,b)|=1$, since the structure constant is unital. According to Cauchy's theorem, if a prime $p$ divides $|G|$ there exist a subgroup $H$ of order $p$ in $G$. If  such a $p$ is odd then the determinants just discussed will be polynomials of odd order in a single variable that have always non trivial roots, leading to zero divisors. The reals can viewed as an $\rr$--algebra
graded by the trivial group $G\!=\!\{e\}$, with $|G|\!=\!2^0\!=\!1$. Recall that if  ${\cal A}$ is finite dimensional over $\rr$, then ${\cal A}$ is a NNA division $\rr$--algebra if and
only if  it has no zero divisors. We just proved a particular case of Hopf's theorem  with elementary tools:
\begin{proposition} A twisted group algebra $({\cal A};G,\rr,\{1,-1\},C)$ with
$|G|$ not a power of 2 has zero divisors, and it it is not a NNA division algebra.
\end{proposition}
 The next  proposition will prove efficient in our classification pursuit.
\begin{proposition} Let $({\cal A};G,\rr,\{1,-1\},C)$ be a twisted group algebra as well as a NNA division $\rr$--algebra,  $\{v_g\ : g\in G\}$ be the basis associated to the structure constant $C$,  $I$ be the identity matrix, $|g|$ be the order of $g\in G$,  $e$  be the identity in $G$. Let  $R(x)$ and $L(x)$ be the arrays associated to $x\in{\cal A}$ in (\ref{DetMRML}). Then:\\
(i) Let $e\neq g\in G$. Then  $L(v_g)^{|g|}=R(v_g)^{|g|}=-I$.\\
(ii) Let $t=0,\cdots, |h|-1$, $e\not\in\{ h,g\}\subset G,\ h\neq g$ and $|h^{-1-t}\,g\,h^t|>1$ (resp. $|h^t\,g\,h^{-1-t}|>1$). Then $(L(v_h)^{|h|-1-t}L(v_g)L(v_h)^t)^{|h^{-1-t}\,g\,h^t|}=-I$ (resp. $(R(v_h)^{|h|-1-t}R(v_g)R(v_h)^t)^{|h^t\,g\,h^{-1-t}|}=-I$).\\
(iii) Let $e\!\not\in\! \{h,g,f\}\!\subset\! G,  h\!\neq\! g\!\neq\! f\neq\! h$, $|h^{-1}g|\!>\!1$ (resp. $|g\,h^{-1}|\!>\!1$), and $|g^{-1}\,f|$ $>\!1$ (resp. $|f\,g^{-1}|\!>\!1$). Then
$((L(v_h)^{|h|-1}L(v_g))^{|h^{-1}g|-1}L(v_h)^{|h|-1}L(v_f))^{|g^{-1}f|}=-I$ (resp. $((R(v_h)^{|h|-1}R(v_g))^{|g\,h^{-1}|-1}R(v_h)^{|h|-1}R(v_f))^{|f\,g^{-1}|}=-I).$\\
\end{proposition}
\noindent {\bf Proof:} (i)  Let $g\neq e$  be of order $|g|$, and $v_g$ be in the basis associated to $C$. Then $(L(v_g))^{2|g|}=I$. Assume that $L(v_g)$ has an Eigenvalue $\lambda\in\rr$ with Eigenvector $x\neq 0$. Then  $(v_g-\lambda\;v_e)\cdot x=0$ leading to zero divisors. Then, $\lambda$ can not be real. Now, $(L(v_g))^{2|g|}-I\!=\!(L(v_g)-1)(L(v_g)+1)((L(v_g)^2)^{|g|-1}+\cdots+L(v_g)^2+1)=0$. Absence of zero divisors forces $(L(v_g)^2)^{|g|-1}+\cdots+L(v_g)^2+1=0$. Reuniting the terms of degree $e$ leads to $L(v_g)^{|g|}+1=0$. Similarly we verify $R(v_g)^{|g|}+1=0$.\\
(ii) From the hypotheses, $h^{-1}\!=\!h^{|h|-1}$, $(L(v_h)^{|h|-1-t}L(v_g)L(v_h)^t)^{2|h^{-1-t}\,g\,h^t|}\!=\!I$. Assume that $L(v_h)^{|h|-1-t}L(v_g)L(v_h)^t$ has an Eigenvalue $\lambda\in \rr$, with Eigenvector $x\!\neq\! 0$. Then $L(v_h)^{|h|-1-t}L(v_g)L(v_h)^t\,x=\lambda\, x=-\lambda\, L(v_h)^{|h|}x$, where we used (i). So $L(v_h)^{|h|-1-t}(L(v_g)+\lambda\,L(v_h))L(v_h)^t\,x=0$. This would lead to zero divisors. Then,  $\lambda$ cannot be real. Analogous arguments as in (i) lead to the result. Part (iii) is proven similarly using (ii) for $t=0$, and it can be extended easily for other values of $t$, and other placements of the factor including $L(v_f)$.$\ \Box$

\section{Standard bases}

All the groups of order 2, 4, and 8 are solvable and thus have composition series where all factors are $\z_2$. This allows for the  recursive definition of certain types of bases in the algebras we want to consider. This a rather technical definition, and the reader might skip it  and return to it upon need.
\begin{definition}\label{StandardBasis}
Let $G\neq \{e\}$ and $({\cal A}; G,\rr, A,$ $C)$ be a twisted group algebra with basis $\{v_g\ :
g\in G\}$ as an $\rr$-- vector space and $v_e$ the algebra unit. Let $H$ be a subgroup of $G$ of index 2. So $H$ is normal in $G$. Let $s\in G-H$. Then $s^2=\hat{h}\in H$ and $G= H\dot{\cup}Hs=H\dot{\cup}sH$. Now, if  there exists $h\in H$ such that $h^2=\hat{h}$ then $s$ is called degenerated. The set $\{v_h\ : h\in H\}\dot{\cup}\{v_h\cdot v_s\ :h\in H\}$ with the {\bf relative normalization of $v_s$ and $v_{\hat{h}}$} given by $v_s\cdot v_s\in\{v_{\hat{h}},-v_{\hat{h}}\}$ when $s$ is degenerated and given by $v_s\cdot v_s=v_{\hat{h}}$ when $s$ is not degenerated  is called a {\bf right--standard basis for ${\cal A}$ with respect to $\hat{\cal A}$} if $\{v_h\ : h\in H\}$
 is itself a right--standard basis for $\hat{\cal A}$ with respect to a subalgebra or $H=\{e\}$. In this case ${\cal A}=\hat{\cal A}\dot{+}(\hat{\cal A}\cdot v_s)$ as a vector space and each $ z\in {\cal A}$ can be uniquely written as $ z=z'+z''\cdot v_s=:(z',z'')_R$ where $z',z''\in \hat{\cal A}$. A right--standard basis is {\bf ordered} when the corresponding basis elements are listed as an $|G|$-tuple $[ v_e, v_{h_1},\cdots, v_{h_{|H|-1}}, v_s,v_{h_1}\cdot v_s, \cdots, v_{h_{|H|-1}}\cdot v_s]$ for an adopted listing $e, h_1,\cdots, h_{|H|-1}$ of the elements of $H$.  An {\bf (ordered)  left--standard basis for ${\cal A}$ with respect to $\hat{\cal A}$} is defined  similarly. In that case ${\cal A}=\hat{\cal A}\dot{+}(v_s\cdot \hat{\cal A})$ as a vector space and each $ z\in {\cal A}$ can be uniquely written as $ z=z'+ v_s\cdot z''=:(z',z'')_L$ where $z',z''\in \hat{\cal A}$.
 Given a right-- (resp. left--) standard basis for ${\cal A}$ with respect to $\hat{\cal A}$, the product in ${\cal A}$ in terms of the pairs  $(z',z'')_R$  (resp. $(z',z'')_L$) is called a right-- (resp. left--) Cayley-Dickson decomposition for ${\cal A}$ with respect to $\hat{\cal A}$. We denote ${\cal A}^L$ (resp. ${\cal A}^R$) the algebra $({\cal A};\ G,\rr, \{1,-1\},$ $C)$ with the adoption of a fixed left-- (resp. right--) standard basis with respect to a fixed $\hat{\cal A}$.
 \end{definition}

{\bf Examples:} A twisted group algebra with grading group  $Z_2\equiv \z/2\z=\{e,a\}$ has an ordered right-- and left--standard basis [generated by $v_a$] with respect to $\rr v_e$ given by  $[v_e,v_a]$ satisfying the relative normalization  $v_a\cdot v_a \in\{v_e,-v_e\}$. If initially $v'_a\cdot v'_a=k\; v_e$  for a $0\neq |k|\neq 1$ we take $v_a=\pm |k|^{-1/2}\; v'_a$,  and the sign freedom can be used in a further normalization of  $v_a$. A twisted group algebra with grading group $\z_4=\{e,b,b^2\equiv a, b^3\}$ has an ordered left-- (resp. right--) standard basis [generated by $v_b$] with respect to $\rr v_e\oplus \rr v_a$  given by  $[v_e, v_a, v_b, v_b\cdot v_a]$ (resp. $[v_e, v_a, v_b,v_a\cdot v_b]$) satisfying the relative normalizations  $v_a\cdot v_a \in\{v_e,-v_e\}$ (inherited from the left--standard basis of $\rr v_e\oplus \rr v_a$) and $v_b\cdot v_b=v_a$.   If we had  $v'_b\cdot v'_b=r\; v'_a$  with $0\neq r\neq 1$ we take  $v_a=(r/|r|)v'_a$ (using the sign freedom just remarked) and $v_b=\pm |r|^{-1/2}\; v'_b$. A further sign freedom remains. We will adopt another ordering for the left-- (right--) standard basis, namely $[v_e,  v_b, v_a, v_b\cdot v_a]$ (resp. $[v_e,  v_b, v_a, v_a\cdot v_b]$) which is closer to the ordering of the group. Now, let $G$ be $\z_2\times\z_4$, or the dihedral group of  eight elements, or the quaternion group. A twisted group algebra with grading group $G$ will have an ordered left- (resp. right--) standard basis [generated by $v_b$ and $v_s$] with respect to a subalgebra with grading group $H=\{e,b,b^2\equiv a, b^3\}\simeq\z_4$ of the form
$[ v_e, v_{b},v_a,v_{b}\cdot v_a,v_s,v_s\cdot v_{b},v_s\cdot v_a,v_s\cdot (v_{b}\cdot v_a)]$ (resp. $[ v_e, v_{b},v_a, v_a\cdot v_b,v_s,v_b\cdot v_{s}, v_a\cdot v_{s}, ( v_a\cdot v_{b})\cdot v_{s}]$) with $s\in G-H$, $s^2=\hat{h}\in H$, and the normalizations  $v_a\cdot v_a \in\{v_e,-v_e\}$,   $v_b\cdot v_b=v_a$ (both  inherited from the left-- (resp. right--) standard basis of the sub\-algebra with grading group $H$), and $v_s\cdot v_s\in \{v_{\hat{h}},-v_{\hat{h}}\}$ since $s$ turns out to be degenerated.

 We remark now, that if we have a twisted group algebra $({\cal A};\ G,\rr, \{1,-1\},$ $C)$, then we can always adopt a left-- (or right--) standard basis whose corresponding structure constant of $G$ is in $\{1,-1\}$ (since not every basis change maintains this feature). The reason for that is simple. Assume we have a basis $\{v'_g| g\in g\}$ for ${\cal A}$ with corresponding unital structure constant $C'$ of $G$ in $\{1,-1\}$.  We adopt $v_e:=v'_{e}=1$ and using the relative normalizations we can construct a  left-- (or right--) standard basis as explained above, whose corresponding structure constant turns out to be also in $\{1,-1\}$. The  terms with factors of the form $v'_{g_i}$, with the $g_i$'s generators of G  used to define the basis element $v_g$ (no matter which configuration of parenthesis) will be reduced to a product of structure constant factors $C'(g_i,g_j)$ or $\pm 1$'s  times a resulting $v'_{g}$. Hence, either  $v_{g}= v'_{g}$ or $v_{g}= -v'_{g}$ occur. The products $ v_g\cdot v_f=C(g,f)v_{gf}$ have a $C(g,f)$ that can differ from $C'(g,h)$ only by a sign. {\bf  Adopting a left-- (or right--) standard basis will constitute no constraint of generality}.

 Let $({\cal A};\ G,\rr, \{1,-1\},$ $C_{\cal A})$ be a twisted group algebra over $\rr$ with $G$ abelian. If $C$ is the structure constant for $G$ in $\{1,-1\}$  for a given ordered normalized left-- (right--) standard basis choice, then the opposite  algebra ${\cal A}^{opp}$ will have a structure constant $C_{{\cal A}^{opp}}$  for $G$ in $\{1,-1\}$ which is the transposed of the array $C_{\cal A}$ when using the corresponding  right-- (left--) standard basis that reverses the order of the factors in the constructed basis for ${\cal A}$.  Observe that the result holds when $G$ is nonabelian, but besides reversing the order of the factors in the basis we replace every appearance $v_g$ by $v_{g^{-1}}$, and besides transposing the array $C$  we replace $g$ by $g^{-1}$ (using the group anti--isomorphism $g\mapsto g^{-1}$) in its labels to obtain $C_{{\cal A}^{opp}}$.  Not every twisted group algebra (even if the grading group is abelian) is  isomorphic to its opposite algebra. That is,  ${\cal A}^{opp}\not\simeq {\cal A}$ can happen.

 Let $\{v_{s_i}|i=1,\cdots,t\}$ be the minimal set of generators of ${\cal A}=\rr_C G$ used to build  the left-- (or right--) standard basis  $\{v_g|g\in G\}$ with $v_{a}\cdot v_b=C(a,b)v_{ab}$.  In each  algebra automorphism mapping $v_{s_i}\!\mapsto v'_{s_i},\ i=1,\cdots,t$,  the left-- (or right--) standard basis $\{v'_g|g\in G\}$ generated by the $v'_{s_i}\!$'s satisfy $v'_{a}\cdot v'_b=C(a,b)v'_{ab}$ with the same structure constant. So, {\bf the group of automorphisms is characterized by the linear bijective mappings from old to new generators (of the whole algebra) preserving the structure constant}.

 Let ${\cal A}$ and ${\cal B}$ be two twisted group algebras both with grading group $G$, both spanned by the same type of left--standard bases $\{v_g|g\in G\}$ resp.  $\{w_g|g\in G\}$ with respect to the same fixed proper subalgebra,   both  subject to the same  constraints arising from proper subalgebras  but  having different structure constants $C_{\cal A}\neq C_{\cal B}$. Now, if instead of the  original  minimal set of generators $\{v_{s_i}| i=1,\cdots,t\}$ for ${\cal A}$, we select  a generic  minimal set of  generating elements  $\{v'_{s_i}| i=1,\cdots,t\}\subset{\cal A}$  subject only to the same constraints arising from proper subalgebras. With them we produce another left--standard basis $\{v'_g|g\in G\}$ with respect to the  fixed proper subalgebra which generate the whole algebra satisfying $v'_{a}\cdot v'_b=C'_{\cal A}(a,b)v'_{ab}$. If the resulting structure constant $C'_{\cal A}$ coincides always with $C_{\cal A}$ for such an arbitrary choice of generators, then the algebra ${\cal A}$ has a unique structure constant for such a type of left--standard bases and therefore ${\cal A}$ is not isomorphic to ${\cal B}$. If such uniqueness could not be established, given $\Phi:G\rightarrow G$ a  nontrivial group automorphism, the adoption of generators $v''_{s_i}:=v_{\Phi(s_i)},\ i=1,\cdots,t$, for ${\cal A}$ generate a left--standard basis $\{v''_g|g\in G\}$ that might have a different structure constant. Obviously, the considerations of this paragraph can be made for right--standard bases. It is frequently demanding to prove or disprove uniqueness of the structure constant  under such generic choice of generators. In that case the problem of establishing isomorphisms between two algebras brings us back to  finding bijective maps $\phi:{\cal A}\rightarrow {\cal B}$, such that $\phi(x\cdot_{\cal A}\, y)=\phi(x)\cdot_{\cal B}\, \phi(y), \forall x\in {\cal A}$. In that endeavor, the left-- (or right--) Cayley--Dickson decomposition of both algebras with respect to a fixed subalgebra is  instrumental, particularly when such subalgebra is unique in both algebras.

\begin{definition} Let $|G|\!>\!1$, and  $({\cal A};G,\rr, \{1,-1\},C_{\cal A})$ be a twisted group algebra with $C_{\cal A}$ the structure constant under an ordered right-- (or left--) standard basis for ${\cal A}$ with respect to $\hat{{\cal A}}$ given by $[ v_e, v_{g_1},\cdots, v_{g_{|G|-1}}]$. With this ordered basis we label the components $x_0,\cdots,x_{|G|-1}$ of  each $x\in{\cal A}$.   We consider  unary operations $x \mapsto \tilde{x}$ in  ${\cal A}$, where $\tilde{x}_0=\!x_{0}$, and $\tilde{x}_i=\!x_{i}$ or $\tilde{x}_i=\!-x_{i}\ \forall i\in\{1,\cdots,|G|-1\}$. They result from a combination of reflections and are called {\bf diagonal (involutive) operations} since clearly $\tilde{\tilde{x}}=\!x,\ \forall x\in {\cal A}$. We will single out diagonal operations of two types: we call {\bf  diagonal conjugations} those  $x\mapsto \bar{x}$ such that either (i) $ x\cdot \bar{x}\in{\cal B}\ \forall x\in{\cal A}$, where ${\cal B}$ is a fixed subalgebra of dimension $|G|/2$ (such as $\hat{{\cal A}}$) so that $\det R(x)=\!\det R_{{\cal B}}(x\cdot \bar{x})$ (or $\det L(x)=\!\det L_{{\cal B}}(x\cdot \bar{x})$) with $R_{{\cal B}}$  (resp. $L_{{\cal B}}$) the operators product from the left (resp. right)  in  ${\cal B}$, and ${\cal B}=\!\rr 1$ or ${\cal B}$ has itself a diagonal conjugation; or (ii) $\det R(x)$ (or $\det L(x)$) can be expressed as a monomial built exclusively in terms of $x$, the unary operation, and the product. Now, a diagonal operation $x\mapsto \tilde{x}$ is called a {\bf diagonal involution} if it is an anti--automorphisms (bijective map with $\widetilde{x\cdot y}=\tilde{y}\cdot\tilde{x},\ \forall x,y\in {\cal A}$).
\end{definition}

\section{Grading Group of order 2}
 We use the group $\z_2:=\z/2\z=(\{0,1\};+)$ with additive notation. We consider the unital
structure constant array, with $C(1,1)=\alpha\in\{1,-1\}$, with organized  left-- and  right-- standard basis $[v_0,v_1]$. We use proposition 2(i),  and from $L(v_1)^2=-I$ we require $\alpha=-1$ in order to have absence of zero divisors. We obtain in this way the complex numbers by
identifying $v_0\equiv 1$ and $v_1\equiv \ip$, since
$\ip^2=v_1\cdot v_1=v_1^2=C(1,1)v_0=-1$. This leads to the multiplication table:
\begin{center}
\begin{tabular}{l||rr}
${\cc} $&$v_0$&$v_1$\\
\hline \hline
$v_0$ &$v_0$&$v_1$\\
$v_1$ &$v_1$&$-v_0$
\end{tabular}\\
{\bf Table I}: Multiplication table of $({\cc};\z_2,\rr, \{1,-1\},$ $C_{\cc})$\end{center}
 Observe that the structure constant $C_{\cal A}$ for a left-- and right--standard basis of a twisted group division algebra $({\cal A}; \z_2,\rr, \{1,-1\},$ $C_{\cal A})$  is unique. ${\cal A}\equiv \cc$ is thus isomorphic to its opposite algebra and it has an involution. But being $G$ abelian and $C_{\cal A}$ symmetric then ${\cal A}$ is commutative. It turns out to be also associative.
Each
$x\in\cc$ is denoted in terms of its components corresponding to its Cayley-Dickson pair. We define also a diagonal conjugation.
\begin{equation}
x\equiv x_0+x_1\,v_1=:[x_0,x_1]_{\cc},\ \bar{x}:=x_0-x_1\,v_1=[x_0,-x_1]_{\cc},\ \forall x\in\cc.\label{CompConj}\end{equation}
We obtain a composition norm $x\mapsto |x|_{\cc}=(x\bar{x})^{\frac{1}{2}}$, where $x\bar{x}=\det L(x)=\det R(x)$. The conjugation is  also a diagonal involution.

\section{Grading Group of order 4}

There are two possible grading groups of order four: $\z_2\times\z_2$,   and
$\z_4$. The former will lead to the quaternions and the latter to a
novel  NNA  division algebra.

\subsection{Grading group: the Klein Group}

We use additive group
notation for the grading group $\z_2\times \z_2$. We adopt for the  twisted group division algebra
$({\cal A};\z_2\times\z_2,\rr,\{1,-1\}$, $C)$  a
right--standard basis $[v_{(0,0)}\equiv 1,\ v_{(1,0)},\ v_{(0,1)},$
$v_{(1,1)}\equiv v_{(1,0)}\cdot v_{(0,1)}]$ to obtain the usual notation for the resulting algebra. From
equation (\ref{(0)}) we have $C((1,0),(0,1))=1$. Such a twisted group algebra
has  three different   $\z_2$-graded twisted group subalgebras,
whose structure constants have to be $C_\cc$ in order to avoid zero divisors. This forces $C((1,0),(1,0))=C((0,1),(0,1))=C((1,1),(1,1))=-1$. Using such a
basis for a NNA division $\rr$--algebra ${\cal A}$, the  unital structure constant array with $\alpha,\ \beta,\
\delta,\ \epsilon,\ \phi\ \in \{1,-1\}$  has to have the from:
\begin{center}\begin{tabular}{l||rrrr}
$C $&(0,0)\!\!&(1,0)\!\!&(0,1)\!\!&(1,1)\!\!\\
\hline \hline
(0,0)&1&1&1&1\\
(1,0)&1&-1&1&$\alpha$\\
(0,1)&1&$\beta$&-1&$\delta$\\
(1,1)&1&$\epsilon$&$\phi$&-1\\
\end{tabular}\\
{\bf Table II}: Structure constant of $G=\z_2\times\z_2$ in
$\{1,-1\}$\end{center}

Using proposition 2(i), the constraints $R(v_g)^2=-I,\forall g\neq e$ lead to $\phi = -1 ,\ \beta\epsilon = -1,\ \delta\alpha = -1$. While $L(v_g)^2=-I,\forall g\neq e$ lead to $\alpha = -1,\ \phi\epsilon = -1,\ \delta\beta = -1$. Hence, necessarily $\alpha = -1,\  \beta = -1,\ \delta = 1,\  \epsilon = 1$, and $\phi = -1$. We
identify $v_{(0,0)}\equiv 1$, $v_{(1,0)}\equiv \ip$, $v_{(0,1)}\equiv {\,\rm j}$, $v_{(1,1)}\equiv {\,\rm k}$ to obtain the quaternion numbers  $\hh^R$. We emphasize that we use a right--standard basis in the usual quaternion notation.
This leads to the multiplication table:
\begin{center}\begin{tabular}{l||rrrr}
$\hh^R $&$v_{(0,0)}$&$v_{(1,0)}$&$v_{(0,1)}$&$v_{(1,1)}$\\
\hline \hline
$v_{(0,0)}$&$v_{(0,0)}$&$v_{(1,0)}$&$v_{(0,1)}$&$v_{(1,1)}$\\
$v_{(1,0)}$&$v_{(1,0)}$&$-v_{(0,0)}$&$v_{(1,1)}$&$-v_{(0,1)}$\\
$v_{(0,1)}$&$v_{(0,1)}$&$-v_{(1,1)}$&$-v_{(0,0)}$&$v_{(1,0)}$\\
$v_{(1,1)}$&$v_{(1,1)}$&$v_{(0,1)}$&$-v_{(1,0)}$&$-v_{(0,0)}$\\
\end{tabular}\\
{\bf Table III}: Multiplication table of $({\hh^R};\z_2\times\z_2,\rr, \{1,-1\},$ $C_{\hh})$\end{center}
The structure constant for a  right--standard basis of a twisted group division algebra $({\cal A}; \z_2\times\z_2,\rr, \{1,-1\},$ $C)$ is unique. Therefore, it is isomorphic to its opposite algebra, that is $(\hh^R)^{opp}=\hh^{L}$. For a   left--standard basis the unique structure constant is given by  $C_{\hh}$ transposed, and thus $\hh^R$ has an involution. But being $\z_2\times\z_2$ abelian and $C_{\hh}$ non--symmetric then ${\cal A}$ is not commutative, but the inverses  turn out to be two--sided. $\hh^R$ turns out to be associative. We will denote each element of the algebra
$\hh^R$ in terms of its components, i.e. as quadruplets of real
numbers. We define also a diagonal conjugation in $\hh^R$:
\begin{eqnarray}
&&x\equiv x_0+x_1\,\ip+x_2\,{\,\rm j\,}+x_3\,{\,\rm k\,}=:[x_0,x_1,x_2,x_3]_{\hh},\\ &&\bar{x}:=x_0-x_1\,\ip-x_2\,{\,\rm j\,}-x_3\,{\,\rm k\,}=[x_0,-x_1,-x_2,-x_3]_{\hh},\ \forall x\in\hh^R.\label{QuatConj}\end{eqnarray}
The conjugation  is also an involution and leads to the definition of a composition norm $x\mapsto |x|_{\hh}=(x\bar{x})^{\frac{1}{2}}$, satisfying $|x|_{\hh}^4=\!\det L(x)=\!\det R(x)\ \forall x\in\hh^R$.

\subsection{Grading group: cyclic of order 4}
 A twisted group division algebra  $({\cal A};\z_4,\rr,\{1,-1\},C)$ as a
vector space has a left--standard basis $[v_0\equiv 1,\ v_1\equiv \w,\ v_2\equiv \w^2,\ v_3\equiv \w\cdot(\w^2)=: \w^3]$. We adopt in this subsection the definition
of cubic powers with this particular configuration of parenthesis:
stacking new factors from the left. As an
algebra, ${\cal A}$ can be generated by  a $v_1\equiv \w$ alone.
 $\w^2$ generates  a  $\z_2$-graded subalgebra, so
we require $\w^2\cdot \w^2=-1$ to avoid zero divisors. For $\alpha,\ \beta,\ \delta,\
\epsilon,\ \phi,\ \omega\ \in \{1,-1\}$, the unital structure constant $C$ for the adopted left--standard basis  has the form:
\begin{center}\begin{tabular}{l||rrrr}
$C $&0&1&2&3\\
\hline \hline
0&1&1&1&1\\
1&1&1&1&$\alpha$\\
2&1&$\beta$&-1&$\delta$\\
3&1&$\epsilon$&$\phi$&$\omega$\\
\end{tabular}\\
{\bf Table IV}: Structure constant of $G=\z_4$ in
$\{1,-1\}$\end{center}

Using proposition 2(i), the constraints $R(v_g)^{|g|}=-I,\forall g\neq e$ lead to $\phi = -1,\ \alpha\delta\omega = -1,\ \beta\epsilon = -1$. While $L(v_g)^{|g|}=-I,\forall g\neq e$ lead to $\delta\beta = -1,\ \epsilon\phi\omega = -1,\ \alpha = -1$. From all these constraints there results two solution sets. The first $\alpha = -1, \beta = 1, \delta = -1, \epsilon = -1, \phi = -1, \omega = -1$ produces an associative algebra since $L(x)R(y)=R(y)L(x)\  \forall x,y$. But the only 4--dimensional associative  division $\rr$--algebra is $\hh$ which can not be generated as an algebra by a single element. The resulting algebra is not a division algebra. In fact, in this algebra  $L(\sqrt{2}+\w-\w^3) R(\sqrt{2}-\w+\w^3)=0$. The other solution $\alpha = -1,\ \beta = -1,\ \delta = 1,\ \epsilon = 1,\ \phi = -1,\  \omega = 1$ leads to a NNA division $\rr$--algebra due to the positive definiteness of
\begin{equation}
 \det R(x)=
(x_0^2+x_2^2)^2+(x_1^2+x_3^2)^2=\det L(x)\label{detMLTess}.\end{equation}
The resulting NNA division $\rr$--algebra is called the {\bf Tesseranion algebra}, denoted by $\te$. In its left--standard basis becomes $\tes$ with multiplication table:
\begin{center}\begin{tabular}{l||rrrr}
$\tes $&$v_{0}$&$v_{1}$&$v_{2}$&$v_{3}$\\
\hline \hline
$v_{0}$&$v_{0}$&$v_{1}$&$v_{2}$&$v_{3}$\\
$v_{1}$&$v_{1}$&$v_{2}$&$v_{3}$&$-v_{0}$\\
$v_{2}$&$v_{2}$&$-v_{3}$&$-v_{0}$&$v_{1}$\\
$v_{3}$&$v_{3}$&$v_{0}$&$-v_{1}$&$v_{2}$\\
\end{tabular}\\
{\bf Table V}: Multiplication table of $(\tes;\z_4,\rr, \{1,-1\},$ $C_{\tes})$\end{center}

The structure constant $C_{\tes}$ for a  left--standard basis of a twisted group division algebra $({\cal A}; \z_4,\rr, \{1,-1\},$ $C)$ is unique. Therefore, it is isomorphic to its opposite algebra, that is $\tes\simeq(\tes)^{opp}=\tesR$. For a   right--standard basis the unique structure constant is given by  $C_{\tes}$ transposed, and thus $\tes$ has an involution.  We will denote each element of $\tes$ in terms of its components,
\begin{equation}
x\equiv x_0+x_1\,\w+x_2\,\w^2+x_3\,\w\cdot \w^2=:[x_0,x_1,x_2,x_3]_{\tes},\ \forall x\in\tes.\end{equation}

Since $\w^2\cdot\w^2=-1$, $\w^2$ generates a subalgebra
isomorphic to $\cc$. We identify for this consideration
$\ip\equiv \w^2$.  Each tesseranion element $x\in\tes$ can be rewritten as a pair of complex numbers using complementary projectors $P_e$ and $P_o$:
\begin{eqnarray}
&& P_{e}(x):=\left( x-\w^2\cdot(x\cdot\w^2)\right)/2 ,\ P_{o}(x):= (I-P_e)(x),\ \ \label{ProjTes}\\
\!\!\!&&x\!=\!
 (x_0+x_2\ip)+\w\cdot(x_1+x_3\ip)\!=\!x_{even}+x_{odd}
\!\equiv\! (X_{Even},\, X_{Odd})_{\tes},\ \ \label{TesDecomp}\\
&&{\rm where}\ x_{even}:= x_0+x_2\,\w^2=P_{e}(x),\ x_{odd}:= x_1\,\w+x_3\,\w^3=P_{o}(x),\nn\\
&& X_{Even}:=[x_0,\ x_2]_{\cc}=P_{e}(x),\ X_{Odd}:= [x_1,\
x_3]_{\cc}=\w^3\cdot P_{o}(x).\label{PairTes}
\end{eqnarray}

\begin{definition} We call resp. $\rr\,1$ {\bf the real subalgebra of $\te$}, $\cc^\te:=\rr\oplus\rr\,\w^2=P_e(\te)$ the {\bf complex subalgebra} or {\bf set of even elements of  $\te$}, and  $\te^{odd}:=\rr\,\w\oplus\rr\,\w^3=\w\cdot\cc^\te=\cc^{\te}\cdot\w=P_o(\te)$  the {\bf set of odd elements of  $\te$}.
\end{definition}
These sets reveal a $\z_2$--grading underlying $\te$: $\cc^\te\cdot \te^{odd}\!=\!\te^{odd}\cdot\cc^{\te}\!=\!\te^{odd},\ \te^{odd}\cdot\te^{odd}\!=\!\cc^{\te}$.
Being $\z_4$ abelian and $C_{\tes}$ non--symmetric then $\tes$ is not commutative.
Now, if $a\in\{1,2,3\}$ then $x\cdot \w^a=\w^a\cdot x$ implies $x\in \rr\oplus\rr\w^a$. So, the {\bf center of the tesseranion is  its real subalgebra  $1\rr$ in $\te$}.

Since $\w\cdot\w^2=-\w^2\cdot\w$, {\bf $\tes$ is not associative, not power associative, not flexible, nor alternative. Hence $\te\not\simeq\hh$}. We verify also that $\te -\{0\}$ is {\bf neither a left-- or right--Bol loop},
 {\bf nor a Moufang loop}.
Recalling the associator $(x,y,z):=(x\cdot y)\cdot z-x\cdot(y\cdot z)$, we find
\begin{eqnarray}
 (x,y,z)\in\te^{odd}\ {\rm and}&
\left((x,y,z)=0\Leftrightarrow x\ {\rm or}\ y\ {\rm or}\ z\in \cc^\te \right),\ \forall x,y,z\in\te;&\ \ \label{AssocTes}\\
& (x'\cdot y')\cdot z'+x'\cdot (y'\cdot z')=0,\  \forall x',y',z'\in\te^{odd}.&\ \ \label{AssocTes1}
 \end{eqnarray}
So,  $\cc^\te$  is the {\bf nucleus of} $\te$, $\te$ is {\bf $\cc$-associative}, and $(\te,+)$  is  a  $\cc^\te$--{\bf bimodule}. Odd elements {\bf anti--associate} among them, so if $x\in\tes-\cc^\te$ then $\w'\!=\!(x,x,x)$ $\neq 0$, $\w'\cdot\w'^2\!=\!-\w'^2\cdot\w'$, and $\w'$ generates  $\tes$. So, {\bf each $x\in\tes-\cc^\te$ generates $\te$}. Hence, {\bf $\cc^\te$ and $\rr\,1$ are the only proper division $\rr$-subalgebras of $\te$}.

The {\bf derivations in $\te$} turn out to be of the form $D(x)\!=\!k[\w^2,x]\!=\!k(\w^2\!\cdot\! x-x\!\cdot\! \w^2),\ k\in\rr\!-\!\{0\}$ and generate a Lie algebra $u(1)$. This conforms to the results in \cite{GMBandJMO} where the derivation algebras of
NNA division $\rr$--algebras have been classified.
 Since $\tes$ is generated by a single element, an automorphism mapping $\phi:\w\mapsto \w'=[a_0,a_1,a_2,a_3]_{\tes}$
satisfies the necessary  condition for the $\z_2$--graded subalgebra
$\phi(\w^2\cdot\w^2)=\w'^2\cdot \w'^2=-1$ which leads to $a_0=a_2=0$ and
$a_1^2+a_3^2=1$. Such $\w'$ generate a  left--standard
basis $[1,\w',\w'^2,\w'\cdot \w'^2]$ with a structure constant coinciding with $C_\tes$. So, no further constraints are required, and this allows to characterize the set of automorphisms.
The same result can be obtained from the $\tesR\simeq\tes$ algebra with its Cayley-Dickson decomposition:
\begin{eqnarray}
x\cdot y&= &(X_{Even}, X_{Odd})_{\tesR}\cdot (Y_{Even}, Y_{Odd})_{\tesR}\nn\\
&=&(X_{Even}Y_{Even}+\ip\,X_{Odd}\overline{Y_{Odd}},X_{Even}Y_{Odd}+
X_{Odd}\overline{Y_{Even}})_{\tesR}.\label{CDTesR}
\end{eqnarray}
Since each automorphism  maps the unique subalgebra isomorphic to $\cc$ into itself, it has the form $\phi:(X_{Even}, X_{Odd})_{\tesR}\mapsto(\phi_1(X_{Even}), \phi_2(X_{Odd}))_{\tesR}$, with $\phi_1,\phi_2:\cc\rightarrow\cc$ bijective and linear, and $\phi_1$ an automorphism of $\cc^{\te}$. Calling $\phi_2(1)=b$ we obtain from $\phi(x\cdot y)\!=\phi(x)\cdot \phi(y)$  that $b\bar{b}=-\ip\phi_1(\ip)$. Hence necessarily $\phi_1(\ip)\!=\ip$,  and $\phi:(X_{Even}, X_{Odd})_{\tesR}\!\mapsto(X_{Even}, b\,X_{Odd})_{\tesR}$ with $b\in\cc$, $b\bar{b}=1$, leading  again to the family of automorphisms $\phi:\w\!\mapsto [0,b_0,0,b_1]_{\tesR}$ with $b_0^2+b_1^2=1$.
Also, transformations of the form $\phi:x\!\mapsto z\cdot(x\cdot z^{-1})$ for $z\in \cc^\te$  with $|z|_{\tes}=\!1$ (taking $z_0^2-z_2^2\!=b_0$, and $2z_0z_2\!=b_1$) lead to the same family of automorphisms. The group of automorphisms of $\tes\simeq\tesR$ is thus $U(1)$ or $S^1$. The decomposition in (\ref{CDTesR}) reveals that $\tes$ corresponds to a modification of the product of $\hh$, whose automorphisms fix  $\ip$. Unlike $\hh$, not every automorphism of a proper $\rr$--subalgebra of $\tes$ (like $\cc^\te$) extends to an automorphism of $\tes$.

We consider the diagonal involutive operation $x\mapsto \bar{x}$ (see definition 6)
 \begin{equation}
 x\mapsto \bar{x}:=[x_0,\,-x_1,\,-x_2,\,-x_3]_{\tes}=(\overline{X_{Even}},\ -X_{Odd})_{\tes}.\label{ConjTes}\end{equation}
It is not an involution (nor an isomorphism), but it is a conjugation since
   $x\cdot \bar{x}=\!\bar{x}\cdot
x=\![x_0^2+x_2^2,0,-x_1^2-x_3^2,0]_{\tes}\in\! \cc^\te$ and $\det R(x)=\!\det L(x)=\!(x\cdot \bar{x})\cdot\overline{(x\cdot \bar{x})}\ \forall x\in\tes$. The map $x\mapsto \bar{x}$ is called {\bf the  (diagonal) conjugation of $\tes$}. We define the
 scalar function $x\mapsto |x|_{\tes}:=\!(\det R(x))^{1/4}$, which satisfies
\begin{eqnarray}
\!\!\!\!~&\!\!\!~&|x|_{\tes}^4=x\cdot\left(\bar{x}\cdot \overline{(x\cdot \bar{x})}\right)=\left(\overline{(x\cdot \bar{x})}\cdot \bar{x}\right)\cdot x
\label{QuartIdTes}\\
&&=x\cdot\overline{((x\cdot\bar{x})\cdot x)}=\overline{(x\cdot(x\cdot\bar{x}))}\cdot x= \overline{(\bar{x}\cdot x^2)}\cdot x= x\cdot \overline{( x^2\cdot\bar{x})}\nn\\
&&=(x\cdot \bar{x})\cdot\overline{(x\cdot \bar{x})}= (x_0^2+x_2^2)^2+(x_1^2+x_3^2)^2=|X_{Even}|_{\cc}^4+|X_{Odd}|_{\cc}^4.\ \ \label{QuartIdTesA}\end{eqnarray}
The determinants in (\ref{detMLTess}) coincide with $|x|_{\tes}^4$. For  $x\neq 0$ and using the identities in (\ref{QuartIdTes}),  we find left--
and right--inverses of $x$ in $\tes$  (which differ in general).
  Since $\tes\!\simeq\!(\tes)^{opp}$, it has an involution. Several maps $x\mapsto\tilde{x}$ where $\tilde{x}_i\!=\!\pm x_{\sigma(i)}$ with $\sigma$ a permutation are  involutions but none is simultaneously a conjugation. Any linear transformation $x\!\mapsto\!\tilde{x}$ satisfying $x\cdot \tilde{x}\in\cc^{\te},\ \forall x\in\tes$ has the form $\tilde{x}\!=\![ax_0+bx_2,-ax_1+bx_3,-ax_2+bx_0,-ax_3-bx_1]_{\tes}$. By requiring $\widetilde{x\cdot y}=\tilde{y}\cdot\tilde{x},\ \forall x,y\in\tes$ we obtain $a\!=\!b\!=\!0$. Also, no linear operation $x\!\mapsto\! \bar{x}$ is an involution where $\det R(x)$ coincides with one of the quartic monomials  between equal signs in (\ref{QuartIdTes}--\ref{QuartIdTesA}). We call  {\bf  the (diagonal) involution  $x\mapsto \tilde{x}$ of $\tes$}:
 \begin{equation}
 x\mapsto \tilde{x}:=[x_0,\,x_1,\,x_2,\,-x_3]_{\tes}=
(X_{Even},\ \overline{X_{Odd}})_{\tes}, \label{AntiInvTes} \end{equation}
 where we used the conjugation in $\cc$. As every involution,
 $\widetilde{x\cdot y}=\!\tilde{y}\cdot\tilde{x}$ and $\tilde{\tilde{x}}=\!x$ $\forall x,y\in\!\tes$. It corresponds to the anti-isomorphism $\stackrel{\sim}{\cdot}\,:\tes\rightarrow(\tes)^{opp}=\!\tesR$. The involution and  the conjugation in $\tes$ are related by:
\begin{equation}
\bar{x}=\w^3\cdot(\tilde{x}\cdot\w),\ \ \
\tilde{x}\ =\ \w^3\cdot(\bar{x}\cdot\w),\ \forall x\in \tes.\label{AntiInvConjTes} \end{equation}

Although, for all $x\in\tes$, we have $x+\bar{x}\in\rr\,1$, we do not
have $x\cdot\bar{x}\in\rr$ in general , but  $x\cdot\bar{x}\in \cc^\te$,
and $|x|_{\tes}>0$ for $x\neq 0$. This suggests the
definition of a {\bf higher order norm}: There is  a conjugation
$x\mapsto\bar{x}$ such that $x\cdot\bar{x}$ belongs to a fixed proper
division subalgebra which has itself a higher
order norm. We proceed in this manner until we arrive
through an even root of a monomial involving $x$ and conjugation(s) to
define a value $|x|\in\rr$ which is an even root of the positive
definite  $\det L(x)$ (or $\det R(x)$). In  $\tes$ the steps are
$\tes\to \cc^\te\to \rr\,1$. The number of steps would give the order of
the norm. So, $|x|_{\tes}$ would be  of 2nd order, while the
norms for the complex and quaternion algebras are first order norms.
The conjugations in the
diverse stages of a higher order norm might not coincide,
although in the case of $\tes$ they do.

In the Appendix A  we prove that $|x|_{\tes}$ satisfies the triangle inequality and is a norm, but it does not satisfy the {\bf composition equality} $(|x\cdot y|=|x||y|\ \forall x,y)$.

We inquire about a generalization of
positive homogeneity $|\alpha\,x|_{\tes}=|\alpha||x|_{\tes}$, $\forall \alpha\in \rr,
x\in \tes$. Although the composition equality  is not satisfied
in general, when any of the
two factors is  even or odd then composition equality holds. This leads to the {\bf pure factor composition equalities}:
\begin{equation}
x\ {\rm or}\ y\in\cc^\te \cup \te^{odd}\Longrightarrow  |x\cdot y|_{\tes}=|x|_{\tes}\,|y|_{\tes}, \ {\rm for \ all}\ \  x,y\in\tes.
\label{SETes}\end{equation}
From this property and from $|\w|_{\tes}=1$ we conclude that all elements in the left--standard basis  of $\tes$ have also unit norm.

We reunite our results thus far in this subsection with the following theorem:

\begin{theorem}
Let ${\cal A}$ be a non necessarily associative division $\rr$--algebra and $G$ a group  with identity $e$, and $2\leq |G|\leq 4$. Let ${\cal A}$ admit a vector space basis $\{v_a | a\in G\}$ such that  $v_a\cdot v_b =C_{{\cal A}}(a,b)\, v_{ab}$ with structure constant $C_{\cal A}$ satisfying  $C_{\cal A}(a,b)\in\{1,-1\}$ and $C_{\cal A}(a,e)=C_{\cal A}(e,b)=1$ for  all $a,b\in G$. That is, ${\cal A}$ is a twisted group algebra $({\cal A}; G,\rr, \{1,-1\},C_{\cal A})$. Then:
\begin{description}
\item{(i)} ${\cal A}$ is $\rr$--algebra isomorphic to one of the mutually non isomorphic  algebras: $\cc, \hh$, or $\te$. Each  ${\cal A}$  as a twisted group algebra spanned under a right-- (resp. left--) standard basis is denoted ${\cal A}^R$ (resp. ${\cal A}^L$) and has a unique structure constant $C_{{\cal A}^R}$ ( resp. $C_{{\cal A}^L}$). Each ${\cal A}$ is isomorphic to its opposite algebra, so ${\cal A}^R\simeq({\cal A}^R)^{opp}={\cal A}^L$. Hence,  $C_{{\cal A}^L}$ is the transpose of $C_{{\cal A}^R}$. Each ${\cal A}$ has a diagonal involution and a diagonal conjugation.
\item{(ii)}  The complex $\cc$ and  the quaternion  algebra $\hh$ are graded by $\z_2$ and $\z_2\times\z_2$ respectively, and  under a right--standard basis $[1,v_1\equiv\ip]$ resp. $[1$, $v_{(1,0)}\equiv\ip,v_{(0,1)}\equiv {\rm j}$, $v_{(1,1)}\equiv \ip\cdot {\rm j}]$ have structure constants underlying multiplication Table I resp. Table III.
    The algebras $\cc$ and $\hh^R$ have diagonal conjugations that are also involutions given by (\ref{CompConj}) and (\ref{QuatConj}) respectively,  and they  lead to norms that make them composition algebras.
\item{(iii)} The tesseranion algebra $\te$  under a  left--standard basis  $[1,v_{1}\equiv\w ,v_{2}\equiv \w^2,v_{3}\equiv \w\cdot \w^2]$ becomes the twisted group algebra $(\tes; \z_4,\rr, \{1,-1\},C_{\tes})$ with its  unique structure constant $C_{\tes}$ underlying multiplication Table V.  The center of $\te$ is its unique real subalgebra $1\rr$. Using the complementary projectors  $P_e$ and $p_o$ in (\ref{ProjTes}), we call $\cc^\te:=P_e(\te)=\rr\oplus\rr\, \w^2\simeq \cc$ the set of even elements or complex subalgebra, and  $\te^{odd}:=P_o(\te)=\rr\,\w\oplus\rr\, \w\cdot \w^2$ the set of odd elements of $\te$. The associators in $\te$ satisfy (\ref{AssocTes}) and $\cc^\te$ is the nucleus, as well as its unique complex subalgebra.  ($\te;+)$ is a $\cc^\te$--bimodule. Odd elements anti--associate among them, see (\ref{AssocTes1}).  $\tes$ admits a diagonal involution in (\ref{AntiInvTes}), and  a  diagonal conjugation in (\ref{ConjTes}) satisfying $x\cdot\bar{x}=\bar{x}\cdot x\in\cc^\te$ for all $x\in\tes$. These operations are related by (\ref{AntiInvConjTes}). No nontrivial linear unary operation $x\mapsto \tilde{x}$ in $\te$ satisfies simultaneously $x\cdot \tilde{x}\in\cc^{\te}$ and $\widetilde{x\cdot y}=\tilde{y}\cdot\tilde{x},\ \forall x,y\in\tes$. The operator  $R(x)$ (resp. $L(x)$) associated with the right-- (resp. left--) multiplication by $x\in\tes$ satisfies $\det R(x)=\det L(x)\ \forall x\in \tes$. We define $x\mapsto\! |x|_{\tes}:=\!(\det R(x))^{1/4}\geq 0$ which is a norm (see Appendix A) satisfying (\ref{QuartIdTes}-\ref{QuartIdTesA}), where $|x|^4_{\tes}$ is expressed as diverse quartic  monomials built with $x$ and the conjugation. The norm does not satisfy the composition equality in general, but it fulfils pure factor composition equalities in  (\ref{SETes}). All left-- (and right--) standard basis elements of $\tes$  have unit norm. $\tes$ is neither  commutative, nor associative, nor flexible, nor power associative, nor alternative.  $\tes -\{0\}$ is neither  a left-- nor right--Bol loop, nor a Moufang loop. $\tes$ is generated  (via its associator) by any single element of $\tes$ not in its nucleus.  The Cayley--Dickson decomposition with respect to $\cc^\te$ of the $\tesR$--product is given in (\ref{CDTesR}). The derivations in $\te$ are of the form $D(x)=k(\w^2\cdot x-x\cdot\w^2),\ k\in\rr-\{0\}$, generating a $u(1)$ Lie algebra. The group of automorphisms of $\te$ constitute a Lie group $U(1)\simeq S^1$. It is the set of transformations $x\mapsto\! z\cdot(x\cdot z^{-1})$ for $z\in\! \cc^\te$ with $|z|_{\tes}=\!1$, or equivalently in $\tesR$, the functions $(X_{Even}, X_{Odd})_{\tesR}\mapsto( X_{Even}, b\,X_{Odd})_{\tesR}$ with $b\in \cc$ and $b\bar{b}=1$. Not every $\rr$--automorphism of $\cc^{\te}\subset\tes$ extends to an $\rr$--automorphism of $\tes$.
\end{description}
\end{theorem}

The  decomposition in (\ref{CDTesR}) characterizes an extension moving from a $\z_2$--grading to a $\z_{4}$--grading. Such a product belongs to a family of generalized Cayley-Dickson doublings considered in \cite{WC}, and \cite{PumplAst} leading to what have been called {\bf nonassociative quaternions}. This family was proven to be equivalent to some 4--dimensional NNA division $\rr$--algebras  already considered in \cite{AHK}, \cite{Albert}, and
 \cite{Dickson}.
 If the algebra demands at least two generators, the name {\it nonassociative quaternion} for it seems consonant. Some NNA  division $\rr$-algebras are generated by a single element (such as $\tes$). From this point of view the name {\it nonassociative quaternions} for them would seem less natural.

We define
$U_{\tes}:=\!\{y\in \tes :|y|_{\tes}=1\}$.  From (\ref{SETes}), if $x$ belongs to $U_{\tes}\cap \cc^{\tes}$ or to   $U_{\tes}\cap{\te}^{odd}$ then $x\cdot U_{\tes}=\!U_{\tes}\cdot x^{-1}=\!(x\cdot U_{\tes})\cdot x^{-1}\!=U_{\tes}$. The image of each element of $U_{\tes}$ under the products by elements of the sets $U_{\tes}\cap \cc^{\te}$ and  $U_{\tes}\cap{\te}^{odd}$  leads to maps involving manifolds homeomorphic to spheres, and to
considerations analogous to the Hopf fibrations  (see \cite{HopfFibr}, \cite{Baez}, \cite{LWTetal}).

We define the {\bf antisymmetric or commutator product} ``$[\cdot,\cdot]$" in $\tes$:
\begin{equation}
[\cdot,\cdot]:\tes\to\tes,\ (x,y)\mapsto[x,y]\equiv
(x\cdot y-y\cdot x)/2.\label{CommTes}\end{equation}
Now, $(\tes, +)$ with the  product ``$[\cdot,\cdot]$" constitutes the
algebra $\tes^-$.
The non-zero commutation relations involving  elements of the basis  $\{ v_0, v_1,v_2,v_3\}$ are:
\begin{equation}
~[v_1,v_2]=v_3,\ \ [v_2,v_3]=v_1,\ \ [v_3,v_1]=v_0.\label{CommTessAlg}\end{equation}
Clearly ${\rm gen}\{v_0,v_1,v_3\}$ is an ideal of $\tes^-$. It is
easy to verify that $\tes^-$ satisfies the Jacobi identity.
Hence, $\tes^-$ {\bf is a Lie
algebra}, and thus {\bf $\tes$ is Lie admissible}. Furthermore, {\bf the ideal  ${\rm
gen}\{v_0,v_1,v_3\}$ as a subalgebra of $\tes^-$  is isomorphic to
the Heisenberg Algebra. Hence, the algebra $\tes$ provides a
one-dimensional representation of the Heisenberg algebra}. The Lie
algebra $\tes^-$ is thus an extension of the Heisenberg algebra,
and it {\bf is solvable but not nilpotent} since its lower central
series stabilizes in the Heisenberg
subalgebra. $\tes^-$ is the 4-dimensional solvable $\rr$--algebra
 $M^{14}_a$ with $a=-1$, of the classification of solvable 4--dimensional algebras
 given by W.A. de Graaf, see \cite{deGraaf}. Hence, $M^{14}_a$ is compatible with a $\z_4$-grading. We remark that $(\tes;+)$ with the product
$x\bullet y\equiv(x\cdot y+y\cdot
x)/2$ is not a Jordan algebra. We summarize:

\begin{theorem}  $\tes$ is Lie--admissible. The resulting Lie algebra $(\tes^-;[\cdot,\cdot])$ is a solvable but not nilpotent extension ($M^{14}_a$, with $a=-1$ using de Graaf's classification, see \cite{deGraaf}) of the Heisenberg algebra. $\tes$ is not Jordan admissible.
\end{theorem}

We found some quartic identities in (\ref{QuartIdTes}--\ref{QuartIdTesA}). We devote attention to some of the identities characterizing $\tes$.  The following two hold in $\tes$, and we call them the {\bf  cubic tesseranity identities  in two variables involving conjugates}:
\begin{equation} \bar{x}\cdot(x\cdot y)=y\cdot(x\cdot \bar{x}),\ \
(y\cdot x)\cdot\bar{x}=(x\cdot \bar{x})\cdot y,\ \ {\rm for\ all}\  x,y\in\tes.\label{Tesseranity1}
\end{equation}
These identities allow us to solve linear equations: $\forall\  x,c,a \in\tes, \ a\neq 0$, then $a\cdot x=c \Longrightarrow x= (\bar{a}\cdot c)
\cdot\overline{(a\cdot\bar{a})}/|a|_{\tes}^4.$ Similarly we solve $y\cdot b=d$ for $b\neq 0$.
\begin{theorem}
The following  inequalities hold for the norm in $\tes\simeq\tesR$. The ones in (\ref{SITes}) are called {\bf composition inequalities}:
\begin{eqnarray}
|x|_{\tes}\leq |x_{even}|_{\tes}+|x_{odd}|_{\tes}\leq 2^{\frac{3}{4}}\,|x|_{\tes}, && \forall x\in\tes,\label{TITes}\\
2^{-\frac{3}{4}}\,|x|_{\tes}| y|_{\tes}\leq |x\cdot y|_{\tes}\leq 2^{\frac{3}{4}}\,|x|_{\tes}| y|_{\tes}, && \forall x,y\in\tes.\label{SITes}
\end{eqnarray}
\noindent {\bf Proof}: When $x\!=\!0$,  (\ref{TITes}-\ref{SITes}) hold. Let $0\!\neq\! x\in\tes$ with $|x_{even}|_{\tes}=\!a\!\geq\! 0$, $|x_{odd}|_{\tes}=\!b\!\geq\! 0$. Then $|x|_{\tes}=\!(a^4+b^4)^{\frac{1}{4}}\!>\!0$. The map $f(a,b)\!=\!(a+b)/(a^4+b^4)^{\frac{1}{4}}$ attains a unique absolute  maximum $f(a,a)=2^{\frac{3}{4}}$ when $a=b$. So $a+b\leq 2^{\frac{3}{4}}|x|_{\tes}$. Now, using  $ x=x_{even}+x_{odd}$ in (\ref{TesDecomp}-\ref{PairTes}) and the triangular inequality, we complete (\ref{TITes}).
We use  the pure factor composition equalities (\ref{SETes}), the triangle inequality, and (\ref{TITes}) to verify the second inequality in (\ref{SITes}):
\begin{equation}
 |x\cdot y|_{\tes}=|x_{even}\cdot y + x_{odd}\cdot y|_{\tes}\leq |x_{even}|_{\tes}| y|_{\tes} + |x_{odd}|_{\tes}| y|_{\tes}\leq 2^{\frac{3}{4}}\,|x|_{\tes}| y|_{\tes}.\nn
 \end{equation}
For verifying the first inequality in (\ref{SITes}) we use (\ref{Tesseranity1}).
Since $x\cdot \bar{x}$ is pure even and $|x\cdot\bar{x}|_{\tes}=|x|_{\tes}^2$, using (\ref{Tesseranity1}) and (\ref{SETes}) we find:
\begin{equation}
 |\bar{x}\cdot(x\cdot y)|_{\tes}=|y\cdot (x\cdot\bar{x})|_{\tes}= |y|_{\tes} |x\cdot\bar{x}|_{\tes}=|y|_{\tes} |x|_{\tes}^2. \label{SITes2}
\end{equation}
Now, using the proved second inequality in (\ref{SITes}) and $|\bar{x}|_{\tes}=|x|_{\tes}$ we find:
\begin{equation}
|\bar{x}\cdot(x\cdot y)|_{\tes}\leq 2^{\frac{3}{4}}\,|\bar{x}|_{\tes}|x\cdot y|_{\tes}=2^{\frac{3}{4}}\,|x|_{\tes}|x\cdot y|_{\tes}.\label{SITes3}
\end{equation}
We combine (\ref{SITes2}-\ref{SITes3}) into $|y|_{\tes} |x|_{\tes}^2\leq 2^{\frac{3}{4}}\,|x|_{\tes}|x\cdot y|_{\tes}$. Now, since $|x|_{\tes}>0$, then $2^{-\frac{3}{4}}\,|x|_{\tes}| y|_{\tes}\leq |x\cdot y|_{\tes}$. This completes (\ref{SITes}).
In diverse applications, taking $2^{\pm 1}$ instead of $2^{\pm\frac{3}{4}}$ in (\ref{TITes}-\ref{SITes}) does the job as well (see\cite{LWTetal}).  $\Box$
\end{theorem}

 There is only one independent cubic multilinear identity fulfilled by $\tes$:
\begin{equation}
(x,y,z)=(z,y,x)\  \ \forall\  x,y,z\ \in\tes.\label{Tesseranity3Lin}
\end{equation}
 The identity (\ref{Tesseranity3Lin}) is called {\bf antiflexibility law} \cite{Roda}, from which Lie--admissibility follows. Antiflexibility is frequently joined with third--power associativity which is not valid in $\tes$.  From (\ref{Tesseranity3Lin}) it follows no non trivial cubic polynomial identities in one variable, and  a unique cubic polynomial identity in two variables not involving conjugates for $\tes$ which we call the {\bf cubic identity in two variables}:
\begin{equation}
(x,x,y)=(y,x,x)\ \   \forall\  x,y\ \in\tes.\label{Tesseranity3}
\end{equation}
There are just two independent {\bf quartic identities in a single variable} (one follows from (\ref{Tesseranity3})):
\begin{equation}(x,x,x^2)=(x^2,x,x)=(x,x^2,x) \ \ \forall\  x\in\tes.\label{Tesseranity4y5}
\end{equation}
We give  just two quartic identities in two variables quadratic in each of them from the family of identities of this kind,  whose left sides coincide (see (\ref{Tesseranity3})):
\begin{eqnarray}
&&(x,x,y^2)=y\cdot(x,x,y)+(x,x,y)\cdot y,\label{Tesseranity22a}\\
&&(y^2,x,x)=y\cdot(y,x,x)+(y,x,x)\cdot y,\ \ \forall x,y\in\tes.\ \ \label{Tesseranity22b}
\end{eqnarray}
All alternative algebras satisfy (\ref{Tesseranity3}--\ref{Tesseranity22b}). A set of quintic multilinear identities just in terms of associators is fulfilled in $\tes$. $\forall x,y,u,v,w\in\tes$:
\begin{equation}
(x,y,(u,v,w))=((x,y,u),v,w)=-(u,(y,x,v),w)=(u,v,(x,y,w)).\label{Tesseranity5lin}
\end{equation}
The identities (\ref{Tesseranity3}) and (\ref{Tesseranity5lin}) hint on a new type of triple system with some resemblance to associative triple system of second kind and Jordan triple systems \cite{Meyberg}. Now, from (\ref{AssocTes}--\ref{AssocTes1}) $\forall q,r,s,t,u,v,x,y,z\in\!\tes$:
\begin{eqnarray}
(r,s,(t,u,v)\cdote(x,y,z))\!=\!(r,(t,u,v)\cdote(x,y,z),s)\!=
 \!((t,u,v)\cdote(x,y,z),r,s)\!=\!0,\label{Tesseranity8lin}\\
 (t,u,v)\cdote((x,y,z)\cdote(q,r,s))=-((t,u,v)\cdote(x,y,z))\cdote(q,r,s).\ \  \label{Tesseranity9lin}
 \end{eqnarray}
The reviewed identities provide a  tool to tell apart extensions. Some identities (or linear combination of them) hold and some don't, depending on the extension. A more systematic study of the identities, as those performed by I. Hentzel, M. Bremner in \cite{HB}  for other algebras, will be addressed for $\te$ and other  NNA division algebras in \cite{LWTetal}. We summarize our results for identities:
\begin{theorem}
$\te$ satisfies the  cubic  identities in (\ref{Tesseranity1}) and (\ref{Tesseranity3}--\ref{Tesseranity3Lin}), the quartic  identities in (\ref{QuartIdTes}--\ref{QuartIdTesA}) and (\ref{Tesseranity4y5}--\ref{Tesseranity22b}), and the multilinear identities (\ref{Tesseranity5lin}--\ref{Tesseranity9lin}).
\end{theorem}

\section{Grading group $G=\z_2\times\z_2\times \z_2$}
We use additive grading group notation. For a
twisted group algebra  $({\cal A};\z_2\times\z_2\times
\z_2,\rr,\{1,-1\},C)$ we adopt an ordered
   right--standard basis with respect to $\hh^R$: $[v_{(0,0,0)}\!\equiv\!1,\ v_{(1,0,0)},\ v_{(0,1,0)},\
v_{(1,1,0)}\!\equiv\! v_{(1,0,0)}\cdot v_{(0,1,0)},\ v_{(0,0,1)}$,
$v_{(1,0,1)}\!\equiv\! v_{(1,0,0)}\cdot v_{(0,0,1)},\ v_{(0,1,1)}\!\equiv\!
v_{(0,1,0)}\cdot v_{(0,0,1) },$ $v_{(1,1,1)}\!\equiv\! v_{(1,1,0)}\cdot
v_{(0,0,1) }]$.

From the adopted basis,
$C((1,0,0),(0,0,1))=1$ and so forth.
 ${\cal A}$ has
three $\z_2\times\z_2$--graded subalgebras generated by subbases,
whose structure constants for an adequate subbasis choice have to be as those for $\hh$ in order to avoid zero divisors. From these constraints, $C$ depends on 18 parameters in $ \{1,-1\}$:
\begin{center}\begin{tabular}{l||rrrr|rrrr}
$C $&(0,0,0)\!\!&(1,0,0)\!\!&(0,1,0)\!\!&(1,1,0)\!\!&(0,0,1)\!\!&
(1,0,1)\!\!&(0,1,1)\!\!&(1,1,1)\!\!\\
\hline \hline
(0,0,0)&1&1&1&1&1&1&1&1\\
(1,0,0)&1&-1&1&-1&1&-1&$\alpha_1$&$\alpha_2$\\
(0,1,0)&1&-1&-1&1&1&$\beta_1$&-1&$\alpha_3$\\
(1,1,0)&1&1&-1&-1&1&$\beta_2$&$\beta_3$&-1\\ \hline
(0,0,1)&1&-1&-1&-1&-1&1&1&1\\
(1,0,1)&1&1&$\delta_1$&$\delta_2$&-1&-1&$\rho_1$&$\rho_2$\\
(0,1,1)&1&$\gamma_1$&1&$\delta_3$&-1&$\omega_1$&-1&$\rho_3$\\
(1,1,1)&1&$\gamma_2$&$\gamma_3$&1&-1&$\omega_2$&$\omega_3$&-1
\end{tabular}\\
{\bf Table VI}: Structure constant of $G=\z_2\times\z_2\times
\z_2$ in $\{1,-1\}$\end{center}
We adopt the basis elements to be labeled by the sub--indices
$0,1,2,3, 4,5,6,7$ instead of $(0,0,0),(1,0,0),(0,1,0),(1,1,0)$,
$(0,0,1)$, $(1,0,1),(0,1,1),(1,1,1)$ respectively. From proposition 2(i) $R(v_g)^2=\!L(v_g)^2=\!-I, \ \forall e\neq g\in G$. This leads to 15 constraints. There remain 3 independent parameters, say $\delta_3, \gamma_2, \gamma_3$. Now, from proposition 2(ii),  $(R(v_{2})R(v_{5}))^2=\!(R(v_{2})R(v_{7}))^2=\!
(R(v_{2})R(v_{4}))^2=\!-I$. From it we obtain $-\delta_3=-\gamma_2=\gamma_3=1$. We find the necessary
constraints:
$-\alpha_1=\alpha_2=-\alpha_3=\beta_1=-\beta_2=\beta_3=-\delta_1=\delta_2=-\delta_3=
\gamma_1=-\gamma_2=\gamma_3=-\rho_1=\rho_2=-\rho_3=\omega_1=-\omega_2=\omega_3=1$.
For the chosen  basis there is a unique  structure
constant  that avoids zero divisors, leading to the multiplication table:
\begin{center}\begin{tabular}{l||rrrr|rrrr}
$\oo^R $&$v_0$&$v_1$&$v_2$&$v_3$&
$v_4$&$v_5$&$v_6$&$v_7$\\
\hline \hline
$v_0$&$v_0$&$v_1$&$v_2$&$v_3$&$v_4$&$v_5$&$v_6$&$v_7$\\
$v_1$&$v_1$&$-v_0$&$v_3$&$-v_2$&$v_5$&$-v_4$&$-v_7$&$v_6$\\
$v_2$&$v_2$&$-v_3$&$-v_0$&$v_1$&$v_6$&$v_7$&$-v_4$&$-v_5$\\
$v_3$&$v_3$&$v_2$&$-v_1$&$-v_0$&$v_7$&$-v_6$&$v_5$&$-v_4$\\ \hline
$v_4$&$v_4$&$-v_5$&$-v_6$&$-v_7$&$-v_0$&$v_1$&$v_2$&$v_3$\\
$v_5$&$v_5$&$v_4$&$-v_7$&$v_6$&$-v_1$&$-v_0$&$-v_3$&$v_2$\\
$v_6$&$v_6$&$v_7$&$v_4$&$-v_5$&$-v_2$&$v_3$&$-v_0$&$-v_1$\\
$v_7$&$v_7$&$-v_6$&$v_5$&$v_4$&$-v_3$&$-v_2$&$v_1$&$-v_0$\\
\end{tabular}\\
{\bf Table VII}: Multiplication table of $({\oo^R};\z_2\times\z_2\times\z_2,\rr, \{1,-1\},$ $C_{\oo^R})$\end{center}
 The resulting unique NNA division $\rr$--algebra is the {\bf octonion algebra $\oo^R$}, with $\oo^R\!\simeq\!(\oo^R)^{opp}\!=\!\oo^L$. The conjugation and involution $x\mapsto \bar{x}\!=\![x_0,-x_1,-x_2$, $-x_3,-x_4, -x_5,-x_6,-x_7]_{\oo^R}$ leads to the composition norm $x\mapsto |x|_{\hh}=(x\bar{x})^{\frac{1}{2}}$.

\section{Grading group $G=\z_2\times\z_4$}
 We adopt $G\!=\!\z_2\times\z_4$ with additive notation. We look for necessary  and sufficient conditions for a twisted group algebra
$({\cal A}; \z_2\times\z_4,\rr, \{1,-1\},C)$ without zero divisors. Also conforming to the definition \ref{StandardBasis},
 an  ordered  left--standard basis with respect to $\tes$
for ${\cal A}$ has the form: $[ v_{(0,0)}\!\equiv\! 1,
v_{(0,1)}\!\equiv\! \w,v_{(0,2)}\equiv\! \w\cdot\w\equiv\!\w^2,
v_{(0,3)}\!\equiv\! \w\cdot\w^2\!\equiv\!\w^3, v_{(1,0)}\!\equiv\!\ipp,
v_{(1,1)}\!\equiv\! \ipp\cdot\w, v_{(1,2)}\!\equiv\! \ipp\cdot \w^2,
v_{(1,3)}\!\equiv\! \ipp\cdot \w^3]$. To alleviate notation, we substitute on occasion the group elements $(0,0),(0,1),(0,2),(0,3),
(1,0),(1,1), (1,2),(1,3)$ (particularly when appearing as sub--indices) by $0,1,2,3,4$, $5,6,7$ respectively, e.g. instead of $y_{(1,2)}$ we write $y_6$. The subalgebras without zero divisors spanned by $[1,\w,\w^2,\w\!\cdot\! \w^2]$ resp. $[ 1,\w^2,\ipp, \ipp\!\cdot\!\w^2]$  are  resp.  $\tes$  and  $\hh^L$ algebras, leading to unique structure constants $C_{\tes}$ resp. $(C_{\hh})^{tr}$. Basis and subalgebra constraints  lead to a $C$  depending on 30 parameters in $\{1,-1\}$:
\begin{center}\begin{tabular}{l||rrrr|rrrr}
$C $&(0,0\!\!&(0,1)\!\!&(0,2)\!\!&(0,3)\!\!&(1,0)\!\!&
(1,1)\!\!&(1,2)\!\!&(1,3)\!\!\\
\hline \hline
(0,0)&1&1&1&1&1&1&1&1\\
(0,1)&1&1&1&-1&$\delta_1$&$\delta_2$&$\delta_3$&$\delta_4$\\
(0,2)&1&-1&-1&1&-1&$\delta_5$&1&$\delta_6$\\
(0,3)&1&1&-1&1&$\delta_7$&$\delta_8$&$\delta_9$&$\delta_0$\\
\hline
(1,0)&1&1&1&1&-1&$\beta_1$&-1&$\beta_2$\\
(1,1)&1&$\gamma_2$&$\gamma_3$&$\gamma_4$&$\alpha_1$&
$\alpha_2$&$\alpha_3$&$\alpha_4$\\
(1,2)&1&$\gamma_5$&-1&$\gamma_6$&1&$\alpha_5$&-1&$\alpha_6$\\
(1,3)&1&$\gamma_8$&$\gamma_9$&$\gamma_0$&$\alpha_7$&
$\alpha_8$&$\alpha_9$&$\alpha_0$\\
\end{tabular}\\
\noindent {\bf Table VIII}: Structure constant for $G=\z_2\times\z_4$ in
$\{1,-1\}$\end{center}

From Proposition 2(i) $R(v_g)^{|g|}=-I, \ \forall e\neq g\in G$. From them we obtain the conditions $\alpha_1\delta_1=\alpha_3\delta_9 = \alpha_7\delta_7= \alpha_9\delta_3= \gamma_9\gamma_3 = \alpha_8\delta_5\alpha_2 = \delta_6\alpha_0\alpha_4 = \gamma_6\gamma_0\gamma_4 = \gamma_8\gamma_5\gamma_2 = \delta_4\alpha_6\delta_0\beta_2 = \delta_8\alpha_5\delta_2\beta_1 = -1$. Now from,  $L(v_g)^{|g|}=-I, \ \forall e\neq g\in G$ we obtain  $\beta_1  = \beta_2  = \alpha_5 \gamma_6  = \alpha_6 \gamma_5  = \delta_6 \delta_5  = \alpha_8 \gamma_9 \alpha_0  = \delta_3 \delta_2 \delta_1 \delta_4  = \delta_7 \delta_8 \delta_9 \delta_0  = \gamma_0 \alpha_7 \gamma_8 \alpha_9  = \gamma_2 \alpha_1 \gamma_4 \alpha_3  = \gamma_3 \alpha_2 \alpha_4  = -1$. Both sets of conditions lead to 20 constraints $\alpha_4  = -\delta_6 \alpha_0 ,\   \alpha_6  = \delta_3 \delta_2 \delta_0 \alpha_1 ,\  \alpha_7  = -\delta_0 \alpha_3 \delta_2 \alpha_5 ,\  \alpha_8  = \delta_6 \alpha_2 ,\ \alpha_9  = -\delta_3 ,\  \beta_1  = -1,\  \beta_2  = -1,\  \delta_1  = -\alpha_1 ,\  \delta_4  = \alpha_1 \delta_3 \delta_2 ,\  \delta_5  = -\delta_6 ,\  \delta_7  = \alpha_3 \delta_2 \alpha_5 \delta_0 ,\  \delta_8  = \delta_2 \alpha_5 ,\  \delta_9  = -\alpha_3 ,\  \gamma_2  = -\gamma_0 \alpha_1 \alpha_5 \alpha_3 ,\  \gamma_3  = \delta_6 \alpha_0 \alpha_2 ,\  \gamma_4  = \alpha_5 \gamma_0 ,\   \gamma_5  = -\delta_0 \alpha_1 \delta_3 \delta_2 ,\ \gamma_6  = -\alpha_5 ,\  \gamma_8  = -\delta_3 \alpha_3 \delta_2 \alpha_5 \gamma_0 \delta_0 ,\     \gamma_9  = -\alpha_2 \delta_6 \alpha_0$, and there remain 10 free parameters. From Proposition 2(ii) $(R(v_{4})R(v_{j}))^4=(R(v_{4})R(v_{k}))^2=(L(v_{4})L(v_{j}))^4=(L(v_{4})L(v_{k}))^2=-I$ for $k\in\{2,6\}$, and  $j\in\{1,3,5,7\}$. This will provide 5 further constraints. From  Proposition 2(ii) $(R(v_{1})^3 R(v_{7}))^2=-I$. This provides a further constraint. Using these 6 new constraints, and in terms of $ \alpha_0, \alpha_1, \alpha_3$, and $\delta_0$ the other  26 parameters become: $\alpha_2  = -\alpha_1 \alpha_3 \alpha_0 ,\  \alpha_4  = \alpha_0 ,\  \alpha_5  = -\alpha_3 \alpha_1 ,\  \alpha_6  = \alpha_3 \alpha_1 ,\  \alpha_7  = -\alpha_3 ,\  \alpha_8  = \alpha_1 \alpha_3 \alpha_0 ,\  \alpha_9  = \alpha_1 ,\  \beta_1  = -1,\  \beta_2  = -1,\  \delta_1  = -\alpha_1 ,\  \delta_2  = -\alpha_3 \alpha_1 \delta_0$, $\delta_3  = -\alpha_1 ,\  \delta_4  = \alpha_1 \alpha_3 \delta_0 ,\  \delta_5  = 1,\  \delta_6  = -1,\  \delta_7  = \alpha_3 ,\  \delta_8  = \delta_0 ,\  \delta_9  = -\alpha_3 ,\  \gamma_0  = -\alpha_0 \delta_0 ,\  \gamma_2 =-\alpha_0 \delta_0 ,\  \gamma_3  = \alpha_1 \alpha_3 ,\  \gamma_4  = \alpha_3 \alpha_0 \delta_0 \alpha_1 ,\  \gamma_5  = -\alpha_1 \alpha_3 ,\  \gamma_6  = \alpha_1 \alpha_3 ,\  \gamma_8  = -\alpha_1 \alpha_3 \delta_0 \alpha_0 ,\  \gamma_9  = -\alpha_1 \alpha_3$. Such relations seem  to satisfy all the conditions arising from Proposition 2. There remain 16 choices to be considered. We perform now concrete computations of the determinant of $R(y)$ for particular $y$ values:
\begin{eqnarray}
&&\left. \det R(y)\right|_{y_0=y_2=y_6=y_5=y_7=0}\nn\\
&&\ \ \ \ =(y_4^4+2(\alpha_3 +\alpha_1 )y_3\, y_1\, y_4^2+y_1^4+2\,y_3^2\,y_1^2+y_3^4)\nn\\
&&\ \ \ \ \ \ \ (y_4^4-2\,\alpha_0\, \delta_0 (\alpha_1 +\alpha_3 )y_3 \,y_1\  y_4^2+
y_1^4+2\,y_3^2\,y_1^2+y_3^4
).\label{MLy4y1y3}
\end{eqnarray} This vanishes for $\alpha_3=\alpha_1$ with $y_3=-\alpha_1\,y_1$, and $y_4=2^{1/2}|\alpha_1\,y_1|$. There remain 8 choices when we adopt the necessary constraint $\alpha_3=-\alpha_1$. Now,
\begin{eqnarray}
&&\left. \det R(y)\right|_{y_2=y_6=y_3=y_5=y_7=0, \alpha_3=-\alpha_1}\nn\\
&&\ \ \ \ =(y_1^8+(2\,y_4^4+4\,(\alpha_0\,\delta_0(\alpha_1-1)+\alpha_1)
\,y_0^2\,y_4^2+2\,y_0^4)\,y_1^4\nn\\
&&\ \ \ \ \ \ +4\,y_4^6\,y_0^2+6\,y_4^4\,y_0^4+4\,y_0^6\,y_4^2+y_4^8+y_0^8).\label{MLy1y0y4}
\end{eqnarray}
For $\alpha_0\,\delta_0=1=-\alpha_1$ with $y_4=-y_0$, and $y_1=2^{1/2}|y_0|$ this determinant vanishes. Hence, the cases $\alpha_1=-1$ with $\delta_0=\alpha_0=\pm 1$ lead to zero divisors. There remain 6 choices left. We inquire the case in which $\delta_0=-\alpha_0$:
\begin{eqnarray}
&&\left. \det R(y)\right|_{y_0=y_2=y_4=y_3=y_7=0, \alpha_3=-\alpha_1, \delta_0=-\alpha_0}\nn\\
&&\ \ \ \ =(y_6^4+2\,(1-\alpha_0 )\,y_1 \,y_5 \,y_6^2+y_5^4+2\,y_1^2\,y_5^2+y_1^4)\nn\\
&&\ \ \ \ \ \ \ (y_6^4-2\,(1-\alpha_0 )\,y_1 \,y_5 \,y_6^2+y_5^4+2\,y_1^2\,y_5^2+y_1^4
).\label{MLy6y5y1}
\end{eqnarray}
This determinant vanishes for $\delta_0=-\alpha_0=1$ and $\alpha_1=\pm 1$ with $y_5=-y_1$, $y_6=2^{1/2}|y_1|$. This parameter choices are disjoint with the cases already excluded. There remain 4 choices left, associated with four algebras $\BB_1^L, \BB_2^L, \BB_3^L$, and $\BB_4^L$, whose structure constant and multiplication table are given by:
\begin{center}\begin{tabular}{l||rrrr|rrrr}
$\BB_i^L$&$v_0$&$v_1$&$v_2$&$v_3$&$v_4$&
$v_5$&$v_6$&$v_7$\\
\hline \hline
$v_0$&$v_0$&$v_1$&$v_2$&$v_3$&$v_4$&$v_5$&$v_6$&$v_7$\\
$v_1$&$v_1$&$v_2$&$v_3$&$-v_0$&$-\alpha_{{1}}v_5$&$\delta_{{0}}v_6$&
$-\alpha_{1}v_7$&$-\delta_{{0}}v_4$\\
$v_2$&$v_2$&$-v_3$&$-v_0$&$v_1$&$-v_6$&$v_7$&$v_4$&$-v_5$\\
$v_3$&$v_3$&$v_0$&$-v_1$&$v_2$&$-\alpha_{{1}}v_7$&$\delta_{{0}}v_4$&
$\alpha_{{1}}v_5$&$\delta_{{0}}v_6$\\
\hline
$v_4$&$v_4$&$v_5$&$v_6$&$v_7$&$-v_0$&$-v_1$&$-v_2$&$-v_3$\\
$v_5$&$v_5$&$-\alpha_{{0}}\delta_{{0}}v_6$&$-v_7$&$-\alpha_{{0}}\delta_{{0}}v_4$&
$\alpha_{{1}}v_1$&$\alpha_{{0}}v_2$&$-\alpha_{{1}}v_3$&$\alpha_{{0}}v_0$\\
$v_6$&$v_6$&$v_7$&$-v_4$&$-v_5$&$v_2$&$v_3$&$-v_0$&$-v_1$\\
$v_7$&$v_7$&$\alpha_{{0}}
\delta_{{0}}v_4$&$v_5$&$-\alpha_{{0}}\delta_{{0}}v_6$&$\alpha_{{1}}v_3$&
$-\alpha_{{0}}v_0$&$\alpha_{{1}}v_1$&$\alpha_{{0}}v_2$
\end{tabular}\\
\noindent {\bf Table XI}: Multiplication table of $({\BB_i^L};\z_2\times\z_4,\rr, \!\{1,-1\},$ $\!C_{\BB_i^L}\!)$, with  $(\alpha_0,\alpha_1$, $\delta_0)\!=\!$
 $(1, 1,-1)$ for $\BB_1^L$,
 $(1,-1,-1)$ for $\BB_2^L$,
 $(1, 1, 1)$ for  $\BB_3^L$,
 $(-1, 1,-1)$ for $\BB_4^L$.\end{center}

Instead of algebras  $\BB_i^L$, for $i=1,2,3,4$, we
speak generically of {\bf $\BB_i^L$--algebras}.  Each $\BB_i^L$--algebra fulfills that $\forall y\in \BB_i^L$, $\det R(y) =\det L(y)$  is positive definite:
\begin{equation}
\left.\det R(y)\right|_{\BB_i^L}\left.=\det L(y)\right|_{\BB_i^L}=\left((y_0^2+y_2^2+y_4^2+y_6^2)^2+
(y_1^2+y_3^2+y_5^2+y_7^2)^2\right)^2.\nn\\
\end{equation}
Hence, {\bf each $\BB_i^L$--algebra is thus  a NNA division $\rr$--algebra}.

\begin{definition} We call $\rr\,1$ {\bf the real subalgebra of $\BB_i^L$}. We call $\hh^{\BB_i}:=\rr\oplus\rr\w^2$ $\oplus\rr\,\ipp\oplus\rr\,\ipp\!\cdot\!\w^2=$ $\rr\oplus\rr\w^2\oplus\rr\,\ipp\oplus\rr\,\w^2\!\cdot\!\ipp$ the {\bf quaternion subalgebra} or {\bf set of even elements of  $\BB_i^L$}. We call  $\BB_i^{odd}:=\rr\w\oplus\rr\w^3\oplus\rr\,\ipp\!\cdot\!\w\oplus\rr\,\ipp\!\cdot\!\w^3=$ $\rr\w\oplus\rr\w^2\!\cdot\!\w\oplus\rr\,\w\!\cdot\!\ipp\oplus\rr\,(\w^2\!\cdot\!\w)\!\cdot\!\ipp$ the {\bf set of odd elements of  $\BB_i^L$}.
\end{definition}

In all $\BB_i^L$--algebras $\w\cdot \w^2=-\w^2\cdot \w$, and $\ipp\cdot(\w\cdot \w^2)=-(\ipp\cdot\w)\cdot\w^2$.  So, $\w^2$ and $\ipp$ are not universally associative. We conclude  easily that  the {\bf center as well as the nucleus} of each $\BB_i^L$--algebra is  $\rr\,1$. The $\BB_i^L$ algebras are neither power associative, nor alternative, nor isomorphic to $\oo$. The $\BB_i$--algebras  are {\bf ``Bi--representable''} since each can be represented in terms of pairs  of $\hh$-- {\bf and} in terms of pairs of $\te$--elements. We could suggest the name of {\bf ``Bison algebras''} for them. We denote each  $\BB_i^L$--algebra element in diverse ways: $\forall x\in\BB_i^L$
\begin{eqnarray}
x&\equiv& x_0+x_1\,\w+x_2\,\w^2+x_3\,\w\cdot \w^2+\ipp(x_4+x_5\,\w+x_6\,\w^2+x_7\,\w\cdot \w^2)\nn\\
&=:&[x_0,x_1,x_2,x_3,x_4,x_5,x_6,x_7]_{\BB_i^L}
=x_{{\rm even}}+x_{{\rm odd}}=x_{{\rm left}}+x_{{\rm right}}\\
&&=(X_{Even},X_{Odd})_{\BB_i^L}=\{ X_{Left},X_{Right}\}_{\BB_i^L},\\
x_{{\rm
even}}&:=&[x_0,0,x_2,0,x_4,0,x_6,0]_{\BB_i^L},\ \ X_{Even}:=[x_0,x_2,x_4,x_6]_{\hh^L},\nn\\
x_{{\rm odd}}&:=&[0,x_1,0,x_3,0,x_5,0,x_7]_{\BB_i^L},\ \
X_{Odd}:=[x_1,x_3,x_5,x_7]_{\hh^L},\nn\\
x_{{\rm
left}}&:=&[x_0,x_1,x_2,x_3,0,0,0,0]_{\BB_i^L},\ \
X_{Left}:=[x_0,x_1,x_2,x_3]_{\tes},\nn\\
x_{{\rm right}}&:=&[0,0,0,0,x_4,x_5,x_6,x_7]_{\BB_i^L},\ \
X_{Right}:=[x_4,x_5,x_6,x_7]_{\tes}.
\end{eqnarray}
Analogous expressions are defined for the elements of $(\BB_i^L)^{opp}$ and $\BB_i^R$ in terms of their normalized right--standard bases. As done for $\tes$, we can find projector operators of even and odd parts. For instance, for $\BB_1^L$, $P_e(x):=\sum_{g\in G}b_g\,U_g(x)$, where $U_g(x)=(v_g\cdot x)\cdot v_{g^{-1}}$,  $b=\frac{1}{2}(0,-1,-2,1,2,1,2,-1)$, and $X_{Odd}=\w\cdot(\w^2\cdot(1-P_e)(x))$. Similar expressions can be obtained for the other algebras.

The sets of even and odd elements reveal a $\z_2$--grading in $\BB_i^L$ algebras. Unlike $\tes$, given $x\in\BB_i^L$ the associator  $(x,x,x)$ is not odd in general. Neither do all odd elements anti--associate among them, but $x^2\cdot x\!=\!-x\cdot x^2\ \forall x\in\BB_i^{odd}$.  $\hh^{\BB_i}$ with the left--standard basis $[1,\w^2,\ipp,\ipp\,\w^2]$  is isomorphic to $\hh^{L}$. The subalgebras generated by any $x\in\BB_i^L$ satisfying $0>x^2\in \rr\,1$ are {\bf subalgebras isomorphic to $\cc$}. Direct computation shows that {\bf all the subalgebras isomorphic to $\cc$ in $\BB_i^L$ are of the form ${\cal A}_\cc=\rr 1\oplus \rr I\simeq\cc$ with $0\neq I  \in \rr \w^2\oplus\rr \ipp\oplus\rr \ipp\!\cdot\!\w^2$}. Since all non real elements in $\hh$ generate a subalgebra isomorphic to $\cc$, we conclude that {\bf $\hh^{\BB_i}$ is the only subalgebra isomorphic to $\hh$ in $\BB_i^L$}.

The even and the odd
elements are not universally associative. Parenthesis
matter in general products with three factors even if two of the factors are in $\hh^{\BB_i}\cup\BB_i^{odd}$.
Multiplication by even elements of $\BB_i^L$ from
the left (or right) does not define a quaternion--module structure  in $\BB_i^L$.
 For all $a,b$ in a fixed subalgebra  ${\cal A}_\cc$ isomorphic to $\cc$  we have $a\cdot(b\cdot x)=(a\cdot b)\cdot x$ and $x\cdot(a\cdot b)=(x\cdot a)\cdot b\ \forall\ x\in \BB_i^L$. Hence, {\bf the $\BB_i^L$ are left and right $\cc$--modules with respect to any of their subalgebras ${\cal A}_\cc\simeq\cc$}. The algebras  {\bf $\BB_1^L$, $\BB_3^L$ and $\BB_4^L$  are  $\cc$--bimodules with respect to any  such ${\cal A}_\cc\simeq\cc$} since $\forall a,b\in {\cal A}_\cc$ we have  $a\cdot(x\cdot b)=(a\cdot x)\cdot b$ for all $x$ in these $\BB_i^L$'s. While $\BB_2^L$ {\bf is a $\cc$--bimodule with respect to the  subalgebras of the form ${\cal A}_\cc=\rr 1\oplus \rr I\simeq\cc$ with $0\neq I  \in \rr\w^2$ or $0\neq I \in \rr\ipp\oplus\rr\ipp\!\cdot\!\w^2$}.

We explore anti--isomorphisms
(bijections reversing factors order). $(\BB_1^L)^{opp}$ with  right--standard basis $[1,\w,\w^2,\w^2\!\cdot\w,
\ipp$, $\w\cdot\!\ipp,\w^2\!\cdot\!\ipp,(\w^2\!\cdot\w)\cdot\!\ipp]$ has structure constant
 $C_{\BB_1^L}$ transposed.
We  move to a  left--standard basis $[1,\w,\w^2, \w\!\cdot\!\w^2,
\ipp$, $\ipp\!\cdot\!\w,\ipp\!\cdot\!\w^2,\ipp\!\cdot\!(\w\!\cdot\!\w^2)] =[1,\w,\w^2,-\w^2\!\cdot\!\w,
\ipp,-\w\!\cdot\!\ipp,-\w^2\!\cdot\!\ipp,(\w^2\!\cdot\!\w)\!\cdot\!\ipp]$, and the resulting structure constant
turns out to be $C_{\BB_3^L}$. So,   $\BB_3^L$ and $(\BB_1^L)^{opp}$
{\bf are $\rr$--algebra isomorphic}, $(\BB_1^L)^{opp}\simeq\BB_3^L$. We have thus an
{\bf anti--isomorphism} $\BB_1^L\!\rightarrow \BB_3^L$:
\begin{equation}
~\!\![x_0,x_1,x_2,x_3,x_4,x_5,x_6,x_7]_{\BB_1^L}\mapsto
[x_0,x_1,x_2,-x_3,x_4,x_5,-x_6,-x_7]_{\BB_3^L}.\label{DAntiInvB13}\
\end{equation}
Similarly we find that {\bf the algebras $\BB_2^L$ and $\BB_4^L$ are  each isomorphic to their own opposite algebras: $\BB_2^L\simeq(\BB_2^L)^{opp}$, and $\BB_4^L\simeq(\BB_4^L)^{opp}$}. Therefore they have involutions.  We call
{\bf the diagonal involution in $\BB_2^L$}:
\begin{eqnarray}
&&x\mapsto\tilde{x}=[x_0,x_1,x_2,-x_3,x_4,x_5,-x_6,-x_7]_{\BB_2^L}.
\label{DAntiInvB2}
\end{eqnarray}
We call {\bf the diagonal involution in $\BB_4^L$}:
\begin{eqnarray}
&&x\mapsto\tilde{x}=[x_0,x_1,x_2,-x_3,x_4,-x_5,-x_6,x_7]_{\BB_4^L}.
\label{DAntiInvB4}
\end{eqnarray}

Let $e_{i,j}$ be the matrix  where the only nonzero entry is a $1$ in the $i^{\text{\tiny th}}$ row, $j^{\text{\tiny th}}$ column. The derivation algebras  ${\rm Der}(\BB_i^L)$ for $i=1,2,3$ have generators:
\begin{eqnarray}
&D_0=& e_{2,4}-e_{4,2}-2\,(e_{5,7}-e_{7,5})-(e_{6,8}-e_{8,6})\ \ {\rm of\ degree} \ (0,2)\in G,\ \ \label{DAlgB123}\\
&D_1=&e_{2,4}-e_{4,2}+(e_{6,8}-e_{8,6})\ \ {\rm of\ degree} \ (0,2)\in G,\ \ \nn\\ &D_2=&e_{2,6}-e_{6,2}-(e_{4,8}-e_{8,4})\ \ {\rm of\ degree} \ (1,0)\in G,\ \ \nn\\ &D_3=&-(e_{2,8}-e_{8,2})-(e_{4,6}-e_{6,4})\ \  {\rm of\ degree} \  (1,2)\in G.\ \ \label{DAlgB123a}
\end{eqnarray}
While the ${\rm Der}(\BB_4^L)$ has generators (again as linear operators):
\begin{eqnarray}
&D_0=& e_{2,4}-e_{4,2}+(e_{6,8}-e_{8,6})\ \ {\rm of\ degree} \ (0,2)\in G,\ \ \label{DAlgB4}\\
&D_1=&e_{2,4}-e_{4,2}-2\,(e_{5,7}-e_{7,5})-(e_{6,8}-e_{8,6})\ \
{\rm of\ degree} \ (0,2)\in G,\ \ \nn\\
&D_2=&e_{2,6}-e_{6,2}-2\,(e_{3,7}-e_{7,3})+(e_{4,8}-e_{8,4})\ \
{\rm of\ degree} \ (1,0)\in G,\ \ \nn\\
&D_3=&-(e_{2,8}-e_{8,2})-(e_{4,6}-e_{6,4})\ \  {\rm of\ degree} \  (1,2)\in G.\ \ \label{DAlgB4a}
\end{eqnarray}
 {\bf For each $\BB_i^L$--algebra, ${\rm Der}(\BB_i^L)\simeq\!u(1)\oplus su(2)$} since $[D_0, D_i]=0$, $[D_i,D_j]=2\,D_k$ for $[i,j,k]=\{[1,2,3],[2,3,1],[3,1,2]\}$. This conforms again to the results in \cite{GMBandJMO}. Hence $\oo\not\simeq\BB_i^L$. In each  ${\rm Der}(\BB_i^L)$ algebra the generators can be written as a quadratic expression in the operators $L(v_g), R(v_g)$ with $g\in G$. E.g., in $\BB_4^L$ we have $D_0=R(\w)R(\w)-L(\w)L(\w)
 +R(\ipp\!\cdot\!\w)R(\ipp\!\cdot\!\w)-L(\ipp\!\cdot\!\w)L(\ipp\!\cdot\!\w)$.

Each $\BB_i^L$ is generated by the elements $v_{(0,1)}\equiv \w$ and
$v_{(1,0)}\equiv\ipp$. We adopt a left--standard basis generated by  generic elements $\w'$ and $\ipp'$  in $\BB_i^L$ satisfying
\begin{equation}
\ipp'\cdot\ipp'=\w'^2\cdot \w'^2=(\ipp'\cdot \w'^2)^2=-1,\ \  \w'\cdot(\w'\cdot \w'^2)=-(\w'\cdot \w'^2)\cdot \w'=-1.\label{condsBi}\end{equation}
They are necessary
constraints for the absence of zero divisors in their proper subalgebras.  If we consider a $\ipp'\in \BB_i^L,\ i\in\{1,2,3\}$, such that  $\ipp'\cdot\ipp'=-1$, we obtain that
$\ipp'=[0,0,b_2,0,b_4,0,b_6,0]_{\BB_i^L}$, with $b_2^2+b_4^2+b_6^2=1$.
If $\w'\in \BB_i^L,\ i\in\{1,2,3\}$ generates a 4--
dimensional subalgebra with $\w'^2\cdot \w'^2=-1$, we obtain
\begin{equation}
\w'=[0,a_1,0,a_3,0,a_5,0,a_7]_{\BB_1^i},\ \ {\rm with}\  \
a_1^2+a_3^2+a_5^2+a_7^2=1,\ i\in\{1,2,3\}.\label{Cons1}
\end{equation}
Such $\w'$   turns out also  to fulfil  $\w'\cdot(\w'\cdot
\w'^2)=-(\w'\cdot \w'^2)\cdot \w'=-1$  and generates a $\tes$--subalgebra. The set of choices  in (\ref{Cons1}) constitute a manifold $S^3$ isomorphic to the Lie group $SU(2)$. If we require that  $(\w'^2\cdot \ipp')^2=-1$   so that $\w'^2$
and $\ipp'$  generate $\hh^{\BB_i^L}$, we obtain for the algebras $\BB_i^L, i\in\{1,2,3\}$,  that
$b_2=0$;    that is
\begin{equation}
\ipp'=[0,0,0,0,b_4,0,b_6,0]_{\BB_i^L},\ \ {\rm with}\ \ i\in\{1,2,3\}\ {\rm and} \ b_4^2+b_6^2=1.\label{Cons2}
\end{equation}
This set of choices builds a manifold $S^1$ homeomorphic to the Lie group $U(1)$. The choices of $b_4,b_6$ in (\ref{Cons2}) are independent of the choices of $a_1,a_3,a_5,a_7$ in (\ref{Cons1}).
Now, using the left--standard basis
$[1,\w',\w'^2,\w'^3,\ipp',$ $\ipp'\,\w',\ipp'\,\w'^2,$ $\ipp'\,\w'^3]$   we compute
$C'_{\BB_i^L}$  for all choices of
$\w'$  and  $\ipp'$ in (\ref{Cons1}-\ref{Cons2}) to obtain
$C'_{\BB_i^L}=C_{\BB_i^L}$ always. So, the algebras $\BB_1^L$, $\BB_2^L$, and $\BB_3^L$ are not isomorphic to any of the other algebras since  under an ordered left--standard basis  with respect to $\tes$ their structure constant is unique. This precludes a basis change from such algebras arriving to a structure constant $C_{\BB_4^L}$.  Hence,  there is {\bf no isomorphisms} between the different $\BB_i^L$--algebras, and {\bf under an ordered left--standard basis with respect to $\tes$ each has a unique structure constant}. We have a pair of algebras non isomorphic to their opposite algebras as announced.

To make precise the form of the automorphisms, we present the {\bf right--Cayley--Dickson decomposition of the  product ``$\cdot_i$'' of the $(\BB_i^L)^{opp}$ algebras with respect to their unique subalgebra $\hh^{\BB_i}$}. Let $(z',z'')_R$ and $(y',y'')_R$ be right--Cayley--Dickson pairs of elements in a  fixed $(\BB_i^L)^{opp}$, then:
\begin{eqnarray}
(z',z'')_R\  \cdot_{1}\  (y',y'')_R&=& (z'\,y'+\overline{y''}\,z''\,\ip,
-y''\,\ip\,z'\,\ip-z''\,\ip\,\overline{y'}\,\ip)_R,\label{CDB1}\\
(z',z'')_R \ \cdot_{3}\  (y',y'')_R&=& (z'\,y'+\ip\,\overline{y''}\,z'',
-y''\,\ip\,z'\,\ip-z''\,\ip\,\overline{y'}\,\ip)_R,\label{CDB2}\\
(z',z'')_R\  \cdot_{2}\  (y',y'')_R&=& (z'\,y'+\overline{y''}\,z''\,\ip,
y''\,z'-z''\,\ip\,\overline{y'}\,\ip)_R,\label{CDB3}\\
(z',z'')_R\  \cdot_{4}\  (y',y'')_R&=& (z'\,y'-\ip\,\overline{y''}\,\ip\,z''\,\ip,
-y''\,\ip\,z'\,\ip-z''\,\ip\,\overline{y'}\,\ip)_R,\label{CDB4}\ \
\end{eqnarray}
with $z', z'', y', y'', \ip\in \hh^{\BB_i}\equiv\hh^R$, where we identified $\w^2\equiv\ip$ using the standard notation in $\hh^R$ (no parenthesis since $\hh^R$ is associative). This reveals the algebras as $\oo^R$ with a modified product from which their automorphisms follow. Recall that $(\BB_i^L)^{opp}=\hh^{\BB_i}\oplus\BB_i^{odd}$ as a vector space. Let $\phi:(\BB_i^L)^{opp}\rightarrow (\BB_i^L)^{opp}$ be an automorphism. Since $\phi$ is a linear bijection and the quaternion subalgebra is unique $\phi(\hh^{\BB_i})=\hh^{\BB_i}$, and  $\phi(\BB_i^{odd})=\BB_i^{odd}$. Hence $\phi((z',z'')_R)=(\phi_1(z'),\phi_2(z''))_R$, with  $\phi_1,\phi_2:\hh^R\rightarrow\hh^R$ both linear bijective transformations. Since $\phi_1$ is an $\hh^R$ automorphism, $\phi_1(z')=c\,z'\,\bar{c}$ for some $0\neq c\in\hh^R$, $c\,\bar{c}=1$. We exemplify the obtention of the automorphisms of the diverse $(\BB_i^L)^{opp}$ algebras using the $(\BB_4^L)^{opp}$ as a concrete example. Let $(z',z'')_R,(y',y'')_R\in (\BB_4^L)^{opp}$:
\begin{eqnarray}
(\phi_1(z'\,y'-\ip\,\overline{y''}\,\ip\,z''\,\ip),
\phi_2(-y''\,\ip\,z'\,\ip-z''\,\ip\,\overline{y'}\,\ip))_R\hspace{3.5cm}&&\nn\\
\ \ =(\phi_1(z'),\phi_2(z''))_R \ \cdot_{4}\ (\phi_1(y'),\phi_2(y''))_R\nn\hspace{4.5cm}&&\\
=(\phi_1(z')\phi_1(y')-\ip\overline{\phi_2(y'')}\ip\phi_2(z'')\ip ,
-\phi_2(y'')\ip\phi_1(z')\ip-\phi_2(z'')\ip\overline{\phi_1(y')}\ip).\ &&\label{(2)}
\end{eqnarray}
Accordingly $-\ip\overline{\phi_2(y'')}\ip\,\phi_2(z'')\,\ip=\phi_1(-\ip\,\overline{y''}\,\ip\,z''\,\ip)=$ $-c\,\ip\,\overline{y''}\,\ip\,z''\,\ip\,\bar{c}$. For $y''=1=z''$ we obtain $\overline{(\phi_2(1)\,\ip\,c)}\,\ip\,(\phi_2(1)\,\ip\,c)=\ip$. Hence
\begin{equation} b=(\phi_2(1)\,\ip\,c)\in\rr\oplus\rr\,\ip\ {\rm with}\  b\,\bar{b}=1.\label{(3)}
 \end{equation}
 From (\ref{(2)}), we have $-\phi_2(y'')\ip\phi_1(z')\ip=\phi_2(-y''\,\ip\,z'\,\ip)$. For $y''=s=-z'$ with $s\in\{\ip,{\rm j}, {\rm k}\}\subset\hh^R$ we obtain $\phi_2(s)=-\phi_2(s\,\ip\,s\,\ip)\,\ip\,c\,s\,\bar{c}\,\ip=
 \phi_2(1)\,\ip\,c\,\ip\,s\,\ip\,\bar{c}\,\ip$, since $s\,\ip\,s\,\ip=1$ when $s=\ip$ and $s\,\ip\,s\,\ip=-1$ otherwise, while $\ip\,s\,\ip=-s$ when $s=\ip$ and $\ip\,s\,\ip=s$ otherwise. Now, using (\ref{(3)}) we conclude $\phi_2(s)=b\,\ip\,s\ip\,\bar{c}\,\ip$ for $s\in\{\ip,{\rm j}, {\rm k}\}$. From the linearity of $\phi_2$ we conclude finally $\phi_2(z'')=b\,\ip\,z''\ip\,\bar{c}\,\ip$. Accordingly, each automorphism  $\phi:\!(\BB_4^L)^{opp}\!\rightarrow\! (\BB_4^L)^{opp}$ has the form:
 \begin{eqnarray}
 \phi((z',z'')_R)=(c\,z'\,\bar{c}\ ,\ b\,\ip\,z''\ip\,\bar{c}\,\ip)_R, \ {\rm \ a}\ (\BB_4^L)^{opp}{\rm -automorphism}\nn\\
 {\rm for}\ 0\neq c\in\hh^R\ {\rm with}\  c\,\bar{c}=1,\ b\in\rr\oplus\rr\,\ip\ {\rm with}\  b\,\bar{b}=1.\label{B4autCD}
\end{eqnarray}
Similarly, each automorphism $\phi:(\BB_i^L)^{opp}\!\rightarrow\! (\BB_i^L)^{opp}$ with $i\in\{1,2,3\}$ fulfils:
\begin{eqnarray}
 \phi((z',z'')_R)=(b\,z'\,\bar{b}\ ,\ c\,b\,z''\,\bar{b})_R, \ {\rm \ a}\ (\BB_i^L)^{opp}{\rm -automorphism},\ i\in\{1,2,3\}\nn\\
 {\rm for}\ 0\neq c\in\hh^R\ {\rm with}\  c\,\bar{c}=1,\ b\in\rr\oplus\rr\,\ip\ {\rm with}\  b\,\bar{b}=1.\label{B123autCD}
\end{eqnarray}
No isomorphism  $\phi:(\BB_i^L)^{opp}\!\rightarrow (\BB_j^L)^{opp}$ with $i\neq j$ with $\phi((z',z'')_R\, \cdot_{i}\,  (y',y'')_R)=\phi((z',z'')_R)\, \cdot_{j}\,\phi((y',y'')_R)$ can exist. For instance, when $i=4$ and $j=1$ we have $\overline{\phi_2(y'')}\,\phi_2(z'')\,\ip=\phi_1(-\ip\,\overline{y''}\,\ip\,z''\,\ip)=$ $-c\,\ip\,\overline{y''}\,\ip\,z''\,\ip\,\bar{c}$. For the case $z''=1=y''$ we have $\overline{\phi_2(1)}\,\phi_2(1)=-c\,\ip\,\bar{c}\,\ip$, which is unimodular. Calling $b=\phi_2(1)$ we conclude $\bar{b}\,b=1$. For the case $y''\!=1$ we obtain $\phi_2(z'')\!=-b\,c\,z''\,\ip\,\bar{c}\,\ip$. Now, from $\phi_2(-y''\,\ip\,z'\,\ip)\!= -\phi_2(y'')\,\ip\,\phi_1(z')\,\ip$ and the obtained $\phi_2$, it follows $c\ip\bar{c}\!=1$ which can not be fulfilled (similar for the other cases $i\neq j$). Unlike $\oo$, not every automorphism of a proper $\rr$--subalgebra of $\BB_i^L$ extends to an automorphism of $\BB_i^L$.
But each auto\-morphism of $\hh^{\BB_4}$ or of the subalgebra generated by $\w$ in $\BB_4^L$ extends to an automorphism of $\BB_4^L$; each automorphism of any subalgebra isomorphic to $\te$ in $\BB_i^L$ extends to an automorphism of $\BB_i^L$ when $i=1,2,3$.

We just verified that
{\bf the group of automorphisms of  each $\BB_i^L$ algebra is  $ U(1)\times SU(2)$}. This does not follow from the derivation algebras alone since discrete factors may appear in the automorphism group (not revealed by the derivation algebra).
We completed the following classification result:
\begin{theorem} Let ${\cal A}$ be a non--necessarily--associative (NNA) division $\rr$--algebra admitting a grading over a group $G$ with identity $e$, $G\simeq \z_2\times\z_2\times \z_2$ or $G \simeq \z_2\times\z_4$. Let ${\cal A}$ admit a  vector space basis $\{v_a | a\in G\}$  such that  $v_a\cdot v_b =$ $C_{{\cal A}}(a,b)\, v_{ab}$ with structure constant $C_{\cal A}$ satisfying  $C_{\cal A}(a,b)\in\{1,-1\}$ and $C_{\cal A}(a,e)$ $=C_{\cal A}(e,b)=1$ for  all $a,b\in G$, that is, ${\cal A}$ is a twisted group algebra  $({\cal A}; G,\rr, \{1,-1\},C_{\cal A})$. Then ${\cal A}$ is $\rr$--algebra isomorphic to one of the following mutually non isomorphic  algebras $\oo^R, \BB_1^L, \BB_2^L, \BB_3^L$ and $\BB_4^L$ described below:
\begin{itemize}
\item Any twisted group algebra $({\cal A};\z_2\times\z_2\times \z_2,\rr,\{1,-1\},C_{\cal A})$
which is NNA division $\rr$--algebra
  is isomorphic to the octonion algebra $\oo^R$. Under any ordered right--standard basis  with respect to $\hh^R$, ${\cal A}$ has a unique
structure constant $C_{\oo^R}$ underlying
Table VII. So, $\oo^R\!\simeq\!(\oo^R)^{opp}\!=\!\oo^L$. The conjugation and involution $x\mapsto \bar{x}\!=\![x_0,-x_1,-x_2,-x_3,-x_4, -x_5,-x_6,-x_7]_{\oo^R}$ leads to the norm $x\mapsto |x|_{\hh}=(x\bar{x})^{\frac{1}{2}}$ making $\oo^R$ a composition algebra.

\item  Each twisted group algebra $({\cal A};\z_2\times\z_4,\rr,$
$\{1,-1\},$ $C)$ which is NNA division
$\rr$--algebra is isomorphic to one of the following
 four mutually non $\rr$--algebra--isomorphic
  twisted group algebras $(\BB_i^L;\z_2\times\z_4,\rr,$
$\{1,-1\},$ $C_{\BB_i^L})$, for $i=1,2,3,4$. Under an ordered left--standard basis with respect to $\tes$ given by $[ v_0,\cdots,v_7]\!=\![1,\w,\w^2,\w\cdot \w^2,\ipp,\ipp\!\cdot\!\w,\ipp\!\cdot\!\w^2$, $\ipp\!\cdot\!(\w\cdot \w^2)]$ each $\BB_i^L$--algebra has a unique structure constant $C_{\BB_i^L}$ underlying Table XI. The center and nucleus of each $\BB_i^L$--algebra is $\rr 1\simeq\rr$.
All the subalgebras ${\cal A}_\cc\simeq\cc$ of $\BB_i^L$ have the form ${\cal A}_\cc=\rr 1\oplus \rr I$ with $0\neq I  \in$ $\rr \w^2\oplus\rr \ipp\oplus\rr\ipp\!\cdot\!\w^2$. All the $\BB_i^L$--algebras are left and right $\cc$--modules with respect to any of their subalgebras ${\cal A}_\cc\simeq\!\cc$. While $\BB_1^L$, $\BB_4^L$, and $\BB_3^L$  are also $\cc$--bimodules, $\BB_2^L$ is a $\cc$--bimodule with respect to subalgebras ${\cal A}_\cc=\rr 1\oplus \rr I$ where $0\neq I  \in \rr \w^2$ or $0\neq I \in \rr \ipp\oplus\rr \ipp\!\cdot\!\w^2$. In each $\BB_i^L$--algebra there is a unique subalgebra isomorphic to $\hh$, given by $\hh^{\BB_i}:=\rr 1\oplus\rr\w^2$ $\oplus\rr\ipp\oplus\rr\ipp\!\cdot\!\w^2\subset\BB_i^L$ called the  quaternion subalgebra or  set of even elements of  $\BB_i^L$.   $\BB_i^{odd}:=\rr\w\oplus\rr\w\!\cdot\!\w^2\oplus\rr\ipp\!\cdot\!\w\oplus\rr
\ipp\!\cdot\!(\w\!\cdot\!\w^2)\subset\BB_i^L$ is called the  set of odd elements, and satisfies $x^2\cdot x=-x\cdot x^2\ \forall x\in\BB_i^{odd}$. Each $\BB_i^L$ has a $\z_2$--grading with respect to their subsets $\hh^{\BB_i}$ and $\BB_i^{odd}$.

\item  $(\BB_1^L)^{opp}\!\simeq \BB_3^L$, so $(\BB_1^L)^{opp}\!=\!\BB_3^R$. An anti--isomorphism $\BB_1^L\!\rightarrow\!\BB_3^L$ is given in (\ref{DAntiInvB13}).  $(\BB_2^L)^{opp}\!=\!\BB_2^R$ and  $(\BB_4^L)^{opp}\!=\!\BB_4^R$, so they have involutions in (\ref{DAntiInvB2}--\ref{DAntiInvB4}). The algebra $Der(\BB_i^L)$ with $i\!=\!1,2,3$ is generated by (\ref{DAlgB123}--\ref{DAlgB123a}). $Der(\BB_4^L)$ is generated by (\ref{DAlgB4}--\ref{DAlgB4a}). In general, $Der(\BB_i^L)\!\simeq\! u(1)\oplus su(2)$.   In $\BB_1^L$, $\BB_2^L$, and $\BB_3^L$ the automorphisms map generators  $v_{(0,1)}\!\equiv\! \w\!\mapsto\! \w'$ and $v_{(1,0)}\!\equiv\!\ipp\!\mapsto\! \ipp'$ fulfilling (\ref{Cons1}--\ref{Cons2}). The product in $(\BB_i^L)^{opp}$ is given by its Cayley--Dickson decompositions with respect to $\hh^{\BB_i}$ in (\ref{CDB1}--\ref{CDB4}). The sets of automorphisms are given in (\ref{B123autCD}) for $(\BB_i^L)^{opp}$ with $i=1,2,3$, and in (\ref{B4autCD}) for $(\BB_4^L)^{opp}$. In general, $Aut(\BB_i^L)\!\simeq U(1)\!\times\! SU(2)$, and not every automorphism in of any proper $\rr$--subalgebra of $\BB_i^L$ extends to be in $Aut(\BB_i^L)$. But each automorphism of $\hh^{\BB_4}$ or of the subalgebra generated by $\w$ in $\BB_4^L$ does extend to an automorphism of $\BB_4^L$; Each automorphism of any subalgebra isomorphic to $\te$ in $\BB_i^L$ extends to an automorphism of $\BB_i^L$ when $i=1,2,3$.
\end{itemize}
\end{theorem}

The
function  $x\mapsto\! |x|_{\BB_i^L}:=(\det R(x))^{1/8}\ \forall x\in\BB_i^L$ is a {\bf norm}, see Appendix A.
Let $\hat{{\cal B}}$  be a subalgebra of $\BB_i^L$ with $\hat{{\cal B}}\simeq\hh$ or $\hat{{\cal B}}\simeq\te$. A diagonal operation $x\mapsto\check{x}$ (see Def. 6) such that  $x\cdot\check{x}\in\hat{{\cal B}}$ $\forall x\in\BB_i^L$  fulfils as well that $(x\cdot\check{x})\cdot\overline{(x\cdot\check{x})}\in\hh^{\BB_i}$ and $(x\cdot\check{x},y\cdot\check{y},z\cdot\check{z})\in\BB_i^{odd}\ \forall x,y,z\in\BB_i^L$. From these constraints, for $\BB_i^L$ with $i=1,3,4$ such an operation $x\mapsto\check{x}$ has to be the function $x\mapsto\!\bar{x}$  in (\ref{ConjBi}), while for $\BB_2^L$ it has to be the function $x\mapsto\!\hat{x}$ in (\ref{ConjB2}):
\begin{eqnarray}
&\bar{x}:=&\!\!\![x_0,-x_1,-x_2,-x_3,-x_4,-x_5,-x_6,-x_7]_{\BB_i^L}\nn\\
&&=(\overline{X_{Even}},-X_{Odd})_{\BB_i^L}=\{ \overline{X_{Left}},-X_{Right}\}_{\BB_i^L},\ i=1, 2, 3, 4.\label{ConjBi}\\
&\hat{x}:=&\!\!\![x_0,-x_1,-x_2,-x_3,x_4,-x_5,x_6,-x_7]_{\BB_2^L}=-((\w^2\!\cdot \bar{x})\!\cdot (\!\ipp\!\cdot \w^2))\!\cdot\! \ipp.\ \  \ \label{ConjB2}
\end{eqnarray}
 In general $x\cdot\bar{x}=\bar{x}\cdot x\in\rr\oplus\w^2\rr\simeq\cc\ \forall x\in\BB_1^L$ or $\BB_3^L$; $x\cdot\bar{x}=\bar{x}\cdot x\in \hh^{\BB_4}\ \forall x\in\BB_4^L$; and
 $x\cdot\hat{x},\ \hat{x}\cdot x\in\hh^{\BB_2}\ \forall x\in\BB_2^L$.
 The function $x\mapsto \bar{x}$  in  (\ref{ConjBi}) is a actually a diagonal conjugation in all $\BB_i^L$--algebras (see Def. 6) since  $|x|_{\BB_i^L}^8=\det R(x)$ and
\begin{equation}
|x|_{\BB_i^L}^4=x\cdot\left(\bar{x}\cdot \overline{(x\cdot \bar{x})}\right)=\left(\overline{(x\cdot \bar{x})}\cdot \bar{x}\right)\cdot x\ \
\forall x\in \BB_i^L,\ \ i=1,2,3,4.
\label{NormBId}\end{equation}
We have also $|x|_{\BB_i^L}^8=\!\det R(x)=\!\det R_{\hh}(x\cdot \bar{x})\ \forall x\in\BB_i^L$ with $i=1,3,4$.
The function $x\mapsto \hat{x}$ is also a diagonal conjugation in  $\BB_2^L$ since $|x|_{\BB_2^L}^8=\!\det R(x)=\!\det R_{\hh}(x\cdot \hat{x})\ \forall x\in\BB_2^L$.
There are further quartic identities  expressing the norm:
\begin{eqnarray}
|x|_{\BB_i^L}^4&\!\!=&\!\!(x\cdot \bar{x})\cdot\overline{(x\cdot \bar{x})}=
\overline{(x\cdot(x\cdot\bar{x}))}\cdot x\ = \overline{(\bar{x}\cdot x^2)}\cdot x\nn\\
&\!\!=&\!\!x\cdot\overline{((x\cdot\bar{x})\cdot x)}\ =\ x\cdot \overline{( x^2\cdot\bar{x})}\ \ \forall x\in \BB_i^L,\ i=1,3,4.\label{QuartIdB134}\\
|x|_{\BB_2^L}^4
&=&(x\cdot \hat{x})\cdot\overline{(x\cdot \hat{x})}=(\hat{x}\cdot
x)
\cdot\overline{(\hat{x}\cdot  x)}\nn\\
&=&\widehat{\left(x\cdot\overline{(\widehat{\hat{x}\cdot
x})}\right)}\cdot x\ =x\cdot
\widehat{\left(\overline{(\widehat{x\cdot\hat{x}})}\cdot
x\right)}\ \ \forall x\in \BB_2^L.\label{QuartIdB2}
\end{eqnarray}

In
general $|x\cdot y|_{\BB_i^L}\neq |x|_{\BB_i^L}|y|_{\BB_i^L}$, since $\tes\subset\BB_i^L$, but even or odd
factors produce strict equalities, leading to {\bf pure factor composition equalities}:
\begin{equation}
\forall\ x,y\in
\BB_i^L,\ x\ {\rm or}\ y\ \in\hh^{\BB_i}\cup\BB_i^{odd}:\ |x\cdot y|_{\BB_i^L}=
| y\cdot x|_{\BB_i^L}=|x|_{\BB_i^L}\,|y|_{\BB_i^L}.\label{SEPure}
\end{equation}
From  $|\w|_{\BB_i^L}=|\ipp|_{\BB_i^L}=1$,  all elements in the   left--standard basis  have unit norm.

A particular case of the identities (\ref{Tesseranity1}) provide cubic identities in one variable involving conjugates that {\bf hold  in all the $\BB_i^L$--algebras}:
\begin{equation}
x\cdot(x\cdot\bar{x})=
\bar{x}\cdot x^2,\ \
(x\cdot\bar{x})\cdot x =
x^2\cdot\bar{x},\ \ \forall x\in \BB_i^L.\label{QubId1}
\end{equation}

The  cubic identities  in two variables involving conjugation in (\ref{Tesseranity1}) for $\tes$ {\bf do not hold  in  any $\BB_i^L$--algebra}. Nevertheless
they hold if  $x$ is  even:
\begin{equation} \bar{x}\cdot(x\cdot y)=y\cdot(x\cdot \bar{x}),\ \
(y\cdot x)\cdot\bar{x}=(x\cdot \bar{x})\cdot y,\ \ \forall x,y\in \BB_i^L,\ x\in\hh^{\BB_i}.\label{PETesty1}
\end{equation}

The identities (\ref{Tesseranity1}) used to obtain the first part of the composition inequalities (\ref{SITes}) for $\tes$ do not hold in any $\BB_i^L$--algebra. Nevertheless,  $\forall x,y\in\BB_i^L$:
\begin{eqnarray}
|\bar{x}\cdot(x\cdot y)|_{\BB_i^L}=|y\cdot(x\cdot\bar{x})|_{\BB_i^L}&& {\rm if}\ i=1,4;\\
|(y\cdot x)\cdot\bar{x}|_{\BB_i^L}=|(x\cdot\bar{x})\cdot y|_{\BB_i^L}&& {\rm if}\ i=3,4;\\
\bar{x}\cdot x=x\cdot\bar{x}\in\hh^{\BB_i},\ \ |x\cdot\bar{x}|_{\BB_i^L}=(|x|_{\BB_i^L})^2 &&{\rm if}\ i=1,3,4.
 \end{eqnarray}
 These properties and (\ref{SEPure}) enough to complete an analogous result for the $\BB_i^L$--algebras with $i=1,3,4$ along the lines of Theorem 3. A partial result for the $\BB_2^L$--algebra is also obtained. The inequalities in (\ref{SIB134}-\ref{SIB2}) will be called  the {\bf composition inequalities} of the corresponding algebras:
\begin{eqnarray}
|x|_{\BB_i^L}\leq |x_{even}|_{\BB_i^L}+|x_{odd}|_{\BB_i^L}\leq 2^{\frac{3}{4}}\,|x|_{\BB_i^L}, && \forall x\in\BB_i^L,\label{TIBi}\\
2^{-\frac{3}{4}}\,|x|_{\BB_i^L}| y|_{\BB_i^L}\leq|x\cdot y|_{\BB_i^L}\leq 2^{\frac{3}{4}}\,|x|_{\BB_i^L}| y|_{\BB_i^L}, && \forall x,y\in\BB_i^L,\ i=1,3,4.\label{SIB134}\\
|x\cdot y|_{\BB_2^L}\leq 2^{\frac{3}{4}}\,|x|_{\BB_2^L}| y|_{\BB_2^L}, && \forall x,y\in\BB_2^L.\label{SIB2}
\end{eqnarray}

Let $U_{\BB_i^L}:=\{y\in \BB_i^L\,:|y|_{\BB_i^L}=1\}$. For each $x$ in $U_{\BB_i^L}\cap \hh^{\BB_i}$ or  $U_{\BB_i^L}\cap{\BB_i}^{odd}$ we have $x\cdot U_{\BB_i^L}=\!U_{\BB_i^L}\cdot x^{-1}\!=(x\cdot U_{\BB_i^L})\cdot x^{-1}=\! U_{\BB_i^L}$. This  leads to
considerations analogous to the Hopf fibrations \cite{LWTetal}.

The {\bf cubic tesseranity identities} not involving conjugates
in (\ref{Tesseranity3Lin}--\ref{Tesseranity3}) are not fulfilled in general by $\BB_1^L$, $\BB_2^L$, and $\BB_3^L$. Now, $\BB_4^L$ {\bf fulfils} (\ref{Tesseranity3}) but not (\ref{Tesseranity3Lin}):
\begin{equation}
(x,x,y)=(y,x,x),\ \  \forall\  x,y\ \in \BB_4^L.\label{Tesseranity3B4}
\end{equation}
$\BB_2^L$ does not satisfy any nontrivial linear combination of the quartic identities (\ref{Tesseranity4y5})  that hold for $\tes$. The other $\BB_i^L$--algebras  satisfy them:
\begin{equation}(x,x,x^2)=(x^2,x,x)=(x,x^2,x),\ \ \forall\  x\in \BB_i^L,\ i=1,3,4.\label{Tesseranity4y5B134}
\end{equation}

Recall the quartic identities (\ref{Tesseranity22a}-\ref{Tesseranity22b}) satisfied by $\tes$. The identity (\ref{Tesseranity22a}) is satisfied in $\BB_1^L$, but not in any of the other $\BB_i^L$--algebras, while the identity (\ref{Tesseranity22b}) is satisfied in $\BB_3^L$ but not in any of the other $\BB_i^L$--algebras. $\BB_4^L$  satisfies the subtraction of  these identities (\ref{Tesseranity22a}-\ref{Tesseranity22b}), while $\BB_2^L$ does not satisfy any nontrivial linear combination of them. None of the $\BB_i^L$--algebras satisfy (\ref{Tesseranity5lin}--\ref{Tesseranity9lin}).
 \begin{eqnarray}
&&(x,x,y^2)=y\cdot(x,x,y)+(x,x,y)\cdot y\ \ \forall x,y\in\BB_1^L;\label{Id22B1}\\
&&(y^2,x,x)=y\cdot(y,x,x)+(y,x,x)\cdot y\ \ \forall x,y\in\BB_3^L;\label{Id22B3}\\
&&(x,x,y^2)-y\cdot(x,x,y)-(x,x,y)\cdot y\nn\\
&&\ \ =(y^2,x,x)-y\cdot(y,x,x)-(y,x,x)\cdot y\ \ \forall x,y\in\BB_4^L.\label{Id22B4}
\end{eqnarray}
The identities fulfilled and those not fulfilled assert again that the $\BB_i^L\not\simeq \BB_j^L$ for $i\neq j$. More general identities satisfied by the $\BB_i^L$--algebras are studied in \cite{LWTetal}.

We consider the algebras $(\BB_i^L)^-$ with product $[x,y]:=(x\cdot y-y\cdot x)/2$. The $\BB_i^L$--algebras are neither Lie-- nor Malcev--admissible, but $\BB_4^L$ is Binary Lie admissible: any pair of elements in $(\BB_4^L)^-$ generate a Lie subalgebra \cite{Maltsev}. In $(\BB_4^L)^-$ both $\{\w,\w^3\}$ and $\{\ipp\w,\ipp\w^3\}$ generate copies of the Heisenberg algebra.

\begin{theorem} Let the $\BB_i^L$--algebras with $i= 1,2,3,4$ be the mutually non isomorphic NNA division $\rr$ algebras which are  twisted group algebras over $\rr$ of the form $(\BB_i^L; \z_2\times\z_4,\rr, \{1,-1\},C_{\BB_i^L})$ described in Theorem 5. Then:
\begin{itemize}
\item In each $\BB_i^L$--algebra  $\det R(x)=\det L(x)\ \forall x\in\BB_i^L$  (see (\ref{DetMRML})) and the function $x\mapsto |x|_{\BB_i^L}=(\det R(x))^{1/8}$ is  a norm (see Appendix A). A diagonal operation $x\mapsto\check{x}$  such that $x\cdot\check{x}$ belongs always to a fixed subalgebra isomorphic to $\hh$ or $\te$ has to be the function $x\mapsto\bar{x}$ in (\ref{ConjBi}) for $\BB_i^L$ with $i=1,3,4$; and it has to be the function $x\mapsto\hat{x}$ in (\ref{ConjB2}) for $\BB_2^L$. The function $x\mapsto\bar{x}$ is actually a conjugation in all the $\BB_i^L$--algebras, and $x\mapsto\hat{x}$ is a conjugation in $\BB_2^L$. In general $x\cdot\bar{x}=\bar{x}\cdot x\in\rr\oplus\w^2\rr\simeq\cc\ \forall x\in\BB_1^L$ or $\BB_3^L$; $x\cdot\bar{x}=\bar{x}\cdot x\in \hh^{\BB_4}\ \forall x\in\BB_4^L$; and
 $x\cdot\hat{x},\ \hat{x}\cdot x\in\hh^{\BB_2}\ \forall x\in\BB_2^L$.  The norm relates to  quartic identities involving conjugates in (\ref{NormBId}--\ref{QuartIdB2}). Each $\BB_i^L$--algebra  satisfies  cubic identities involving conjugates  in (\ref{QubId1}-\ref{PETesty1}). The pure factor composition equalities in (\ref{SEPure}) hold in each $\BB_i^L$--algebra. The inequalities in (\ref{TIBi}) hold in each $\BB_i^L$--algebra. The composition  inequalities  in (\ref{SIB134}) hold for $\BB_i^L$ with $i=1,3,4$; the one in (\ref{SIB2}) holds for $\BB_2^L$.

\item The algebra $\BB_4^L$  satisfies the cubic tesseranity identity in two variables not involving conjugates in (\ref{Tesseranity3B4}) while none of the other $\BB_i^L$'s satisfies it. The algebras $\BB_1^L, \BB_3^L$, and $\BB_4^L$ satisfy the quartic identities in one variable not involving conjugates in (\ref{Tesseranity4y5B134}) while $\BB_2^L$ does not satisfy any nontrivial linear combination of them. The identity (\ref{Id22B1}) is satisfied in $\BB_1^L$, but not in any of the other $\BB_i^L$--algebras, while the identity (\ref{Id22B3}) is satisfied in $\BB_3^L$ but not in any of the other $\BB_i^L$--algebras. $\BB_4^L$ fulfills (\ref{Id22B4}). The $\BB_i^L$--algebras are neither Lie-- nor Malcev--admissible, but $\BB_4^L$ is Binary Lie admissible.
\end{itemize}
\end{theorem}

\section{Conclusions and further applications}

The graph II summarized the results of this paper. Further  classification tasks are completed in \cite{PapII}--\cite{PapIII}. Several developments already carried out for the alternative NNA division algebras will be addressed in the paper series \cite{LWTetal}. A question to be addressed there will be which  identities characterize the new algebras. In particular, which identities provide the basic tools to perform a wide variety  of algebraic manipulations in the corresponding algebra (such as those required for solving linear equations, solving geometrical or analytical constructions, etc.). Possed in a different way,  we ask which are the generators of an ideal of the free nonassociative $\rr$--algebra in some variables so that the quotient by this ideal provides the considered algebra.
The  $\rr$--algebra $\cc$
 could  be characterized as the commutative division $\rr$--algebra of lowest dimension having solutions for the equation $x^2+1=0$
 (beyond the possibilities of $\rr$), and with it
 providing an algebraically closed field. Furthermore, $\cc\simeq\rr[x]/(x^2+1)$. The quaternions and octonions
 can solve the equation in two variables $(xy-yx)^2+4=0$ (beyond
 the possibilities of $\cc$). Now, $\te$ and the $\BB_i$--algebras solve the equation $(x\cdot x^2-x^2\cdot x)x-2=0$ (e.g., $x=\w$) which can not be solved
 in any alternative $\rr$--algebra.  $\BB_2^R$  can solve $(x\cdot(x\cdot x^2)+(x^2\cdot x)\cdot x -2\,x^2\cdot x^2)^2+64=0$ (e.g., $x=\w+\ipp$), but no algebra satisfying  identities (\ref{Tesseranity4y5}) or (\ref{Tesseranity4y5B134}) can solve it.

 Some of the identities fulfilled by an algebra have an origin in  identities fulfilled by some functions of its structure constant characterizing its nonassociativity or its noncommutativity \cite{LWTMScthesis}-\cite{SPJMS}. Identities relating to cohomological properties of such functions will be tackled in \cite{LWTetal}. We study  there also
discrete versions of the obtained algebras (analogous to
Lipschiftz and Hurwitz integer quaternions, and octavian integer octonions, see \cite{Conway}). The Cayley--Dickson decompositions of the obtainded NNA division algebras provide a frame  characterizing certain families of NNA division algebra extensions. An examination of such extensions of division algebras, their automorphisms, and their use to represent some NNA algebras and loops is addressed in \cite{LWTetal}.

A relation has been established between the
normed division algebras, their
triality maps, with  spinors and super Yang--Mills theories (and super--gravity theories) in
3, 4, 6, or 10 Minkowski space-time dimensions (see
\cite{BaezSusy} and references therein). A
$\z_4\times\z_4$--graded extension of the Poincar\'e algebra called
the clover extension was introduced in \cite{Lawt1}-\cite{Lawt5}, which is analogous
to supersymmetry.
It led to the problem of finding  division algebras corresponding to the novel space-time
symmetries. The classified NNA division algebras pose the reciprocal task of
determining whether there are space-time symmetries related to them. The algebraic structure behind the group representations in the Standard Model of particle physics is a subject of renewed interest, for a mathematician's account see \cite{BaezHuerta}. We remark that the group of automorphisms of the $\BB_i$--algebras coincides with the internal symmetries in electro--weak interactions. For the use of the novel algebras in (projective) geometry and physics see \cite{LWTetal}.\\

\noindent {\bf Acknowledgements:} This work acquired maturity and some conciseness thanks to  profound recommendations of a referee. E.g. one such a recommendation led to Proposition 2(i), another one led to identities (\ref{Tesseranity8lin}--\ref{Tesseranity9lin}).

\appendix
\section{ Family of norms}
We test the composition equality ($|x\cdot y|=|x|\,|y|\ \forall x,y$) in $\tes$:
\begin{eqnarray}
p:=[1,1,0,0]_{\tes},\ q:=[1,-1,0,0]_{\tes},\ s:=[1,1,1,0]_{\tes},\
t:=[1,-1,1,0]_{\tes},\nn\\
|p|_{\tes}\,|p|_{\tes}\!-|p\cdot p|_{\tes}=-16,\
|p|_{\tes}\,|q|_{\tes}\!-|p\cdot q|_{\tes}=0,\ |s|_{\tes}\,
|t|_{\tes}\!-|s\cdot t|_{\tes}=8.\nn
\end{eqnarray}
The {\bf composition inequality} is {\bf not fulfilled}. We test the triangle inequality:
\begin{eqnarray}
&&\!\!\!|p|_{\tes}+|p|_{\tes}-|p+p|_{\tes}=0,\
|p|_{\tes}+|q|_{\tes}-|p+ q|_{\tes}=2\,(2^{1/4}-1)>0,\nn\\
&&\!\!\!|s|_{\tes}+|t|_{\tes}-|s+t|_{\tes}=2\,(5^{1/4}-2^{1/2})>0.\nn\end{eqnarray}
We will define iteratively a family of scalar functions and prove they are norms. This determines norms for $\tes$ and the $\BB_i^L$-- NNA division $\rr$--algebras:
\begin{theorem} The following  functions $M_j$ for $j=2,3,\cdots$, are norms
\begin{eqnarray}
M_1(a^{(1)})&\!\!\!\!:=&\!\!\!(a_1^2+a_2^2+\cdots+a_n^2)^{\frac{1}{2}},\\
M_2((a^{(1)},a^{(2)}))&\!\!\!\!:=&\!\!\!(M_1(a^{(1)})^4+M_1(a^{(2)})^4)^{\frac{1}{4}},\label{M2Def}\\
\!\!\!\!M_3(((a^{(1)},a^{(2)}),(a^{(3)},a^{(4)})))&\!\!\!\!:=&\!\!\!(M_2((a^{(1)},a^{(2)}))^8\!+\!
M_2((a^{(3)},a^{(4)}))^8)^{\frac{1}{8}},\ \ \\
\cdots,\ M_j((u,r))&\!\!\!\!:=\!\!\!&(M_{j-1}(u)^{2^j}+M_{j-1}(r)^{2^{j}})^{1/2^{j}},\ \ \label{MjDef}\\
{\rm where}\  a^{(1)}=(a_1,a_2,\cdots, a_n)&,&a^{(2)}=(a_{n+1},a_{n+2},\cdots,
a_{2n})\ ,\cdots,\nn\nn\\
a^{(m)}&\!\!\!\!=&\!\!\!(a_{n(m-1)+1},a_{n(m-1)+2},\cdots,
a_{nm})\ {\rm for}\ m=2^j.\nn
\end{eqnarray}
\end{theorem}
\begin{proof} $M_2, M_3,\cdots$ fulfil
positive homogeneity
$M_j(\alpha\,x)=|\alpha|M_j(x)$, and positive definiteness
$(M_j(x)\!=\!0\Longrightarrow x\!=\!0)\ \forall \alpha\in \rr,
x\in \rr^{n\,2^{j-1}}$. We prove by induction
\begin{equation}
M_j(x+y)\leq  M_j(x)+M_j(y),\    \forall x,y\in \rr^{n\,2^{j-1}} \
{\rm (triangle\ inequality)}.\label{TriaIneq}
\end{equation}
 We quote first a particular version of Holder's inequality:
 \begin{equation}
|\rho_1||\xi_1|+|\rho_2||\xi_2|\leq
(|\rho_1|^p+|\rho_2|^p)^{\frac{1}{p}}\,(|\xi_1|^q+|\xi_2|^q)^{\frac{1}{q}},\label{HoldIneq}
\end{equation}
for $\rho_1,\rho_2,\xi_1,\xi_2\in \rr$, $\frac{1}{p}+\frac{1}{q}=1$, and $p,q\in \rr^{+}$.
$M_1$ is
an Euclidean norm. From its triangle inequality and from (\ref{M2Def}), we obtain
\begin{eqnarray}
&&M_1(a^{(1)}+b^{(1)})^4\leq(M_1(a^{(1)})+M_1(b^{(1)}))^4,\\
&&M_2((a^{(1)}+b^{(1)},a^{(2)}+b^{(2)}))^4=M_1(a^{(1)}+b^{(1)})^4+M_1(a^{(2)}+b^{(2)})^4\nn\\
&&\leq
(M_1(a^{(1)})+M_1(b^{(1)}))^4+(M_1(a^{(2)})+M_1(b^{(2)}))^4\nn\\
&&=
M_1(a^{(1)})^4+M_1(a^{(2)})^4+M_1(a^{(2)})^4+M_1(b^{(2)})^4\nn\\
&&\ \ \ +4[M_1(a^{(1)})^3\,M_1(b^{(1)})+
M_1(a^{(2)})^3\,M_1(b^{(2)})]\nn\\
&&\ \ \ +6[M_1(a^{(1)})^2\,M_1(b^{(1)})^2+
M_1(a^{(2)})^2\,M_1(b^{(2)})^2]\nn\\
&&\ \ \ +4[M_1(a^{(1)})\,M_1(b^{(1)})^3+
M_1(a^{(2)})\,M_1(b^{(2)})^3].\label{Lines}
\end{eqnarray}
We take $|\rho_i|=M_1(a^{(i)})^3$, $|\xi_i|=M_1(b^{(i)})$, for
$i=1,2$ and use Holder's inequality (\ref{HoldIneq}) for
$p=\frac{4}{3}$ and $q=\frac{4}{1}$ in the first factor in square
brackets in (\ref{Lines}):
\begin{eqnarray}
&&[M_1(a^{(1)})^3\,M_1(b^{(1)})+
M_1(a^{(2)})^3\,M_1(b^{(2)})]\nn\\
&&\leq(M_1(a^{(1)})^4+M_1(a^{(2)})^4)^{\frac{3}{4}}\,(M_1(b^{(1)})^4+M_1(b^{(2)})^4)^{\frac{1}{4}}\nn\\
&&\ \ =M_2((a^{(1)},a^{(2)}))^{3}\,M_2((b^{(1)},b^{(2)})).
\end{eqnarray}
Similar inequalities are obtained for the second and third factor
in square brackets in (\ref{Lines}), and we can write finally:
\begin{equation}
M_2((a^{(1)}+b^{(1)},a^{(2)}+b^{(2)}))^4\leq [M_2((a^{(1)},a^{(2)}))+M_2((b^{(1)},b^{(2)}))]^4,\label{TesNorm}
\end{equation}
from which the inequality
(\ref{TriaIneq}) for $j=2$ follows. We assume (induction
hypothesis) that (\ref{TriaIneq}) holds. Using it and the definition of $M_{j+1}$ we obtain
\begin{eqnarray}
&&M_j(u+v)^{2^{j+1}}\leq(M_j(u)+M_j(v))^{2^{j+1}},\\
&&M_{j+1}((u+v,w+x))^{2^{j+1}}=M_j(u+v)^{2^{j+1}}+M_j(w+x)^{2^{j+1}}\nn\\
&&\ \ \ \leq
(M_j(u)+M_j(v))^{2^{j+1}}+(M_j(w)+M_j(x))^{2^{j+1}}\nn\\
&&\ \ \ =
M_j(u)^{2^{j+1}}+M_j(v)^{2^{j+1}}+M_j(w)^{2^{j+1}}+M_j(x)^{2^{j+1}}\nn\\
&&\ \ \ \ \ \ +\binom{2^{j+1}}{1}[M_j(u)^{2^{j+1}-1}\,M_j(v)+
M_j(w)^{2^{j+1}-1}\,M_j(x)]+\cdots\nn\\
&&\ \ \ \ \ \  +\binom{2^{j+1}}{2^{j+1}-1}[M_j(u)\,M_j(v)^{2^{j+1}-1}+
M_j(w)\,M_j(x)^{2^{j+1}-1}]\ \ \label{Lines1}\\
&&\ \ \ \leq M_{j+1}((u,w))^{2^{j+1}}+M_{j+1}((v,x))^{2^{j+1}}\nn\\
&&\ \ \ \ \ \ +\binom{2^{j+1}}{1}[M_{j+1}((u,w))^{2^{j+1}-1}\,M_{j+1}((v,x))]+\cdots\nn\\
&&\ \ \ \ \ \ +\binom{2^{j+1}}{2^{j+1}-1}[M_{j+1}((u,w))\,M_{j+1}((v,x))^{2^{j+1}-1}]\ \ \label{Lines2}\\
&&\ \ \ =[M_{j+1}((u,w))+M_{j+1}((v,x))]^{2^{j+1}},
\end{eqnarray}
where we used
the Holder's inequality for each square bracketed term in order
to move from (\ref{Lines1}) to (\ref{Lines2}). So, the triangle inequality holds for $M_{j+1}$.
\end{proof}


\label{lastpage}

\end{document}